\documentclass{scrartcl}
\usepackage[T1]{fontenc}
\usepackage[utf8]{inputenc}
\usepackage{amsmath, amsthm, amssymb}
\usepackage{dsfont}
\usepackage{graphicx}
\usepackage[format=plain]{caption}%ben?tigt, um mit dem Befehl \caption*{} eine Bildunterschrift zu setzen, die nicht mit Abbildung X eingeleitet wird. 'normal'=\setcapindent{0}
\usepackage{capt-of}%ben?tigt, um 'figure' und 'table' nebeneinander zu setzen, siehe 'http://www.faqs.org/faqs/de-tex-faq/part6/' Punkt 6.1.7, aufgerufen am 01.08.11.
\usepackage{subcaption}
\usepackage{bm}
\usepackage{undertilde}

\newcommand{\norm}[1]{\left\| #1 \right\|}  %Norm
\newcommand{\scprd}[1]{\left\langle #1 \right\rangle}  %Skalarprodukt
\renewcommand{\d}{\,\mathrm{d}} %Differenzial in Integralen
\newcommand{\e}{\mathrm{e}} %eulersche Zahl
\newcommand{\N}{\mathbb{N}}  %Zahlbereiche

\newcommand{\R}{\mathbb{R}}
\newcommand{\C}{\mathbb{C}}
\newcommand{\K}{\mathbb{K}}
\newcommand{\eps}{\varepsilon}
\renewcommand{\phi}{\varphi}
\newcommand{\ul}{\underline}
\newcommand{\ol}{\overline}

\renewcommand{\Re}{\operatorname{Re}}

\numberwithin{equation}{section}
\newtheorem{thm}{Theorem}[section]
\newtheorem{cor}[thm]{Corollary}

\newtheorem{lm}[thm]{Lemma}

\theoremstyle{definition} \newtheorem{ex}[thm]{Example}
\theoremstyle{definition}

\title{Stabilization of port-Hamiltonian systems with discontinuous energy densities}

\author{Jochen Schmid\\  
\small Institut f\"ur Mathematik, Universit\"at W\"urzburg, 97074 W\"urzburg, Germany\\
\small jochen.schmid@mathematik.uni-wuerzburg.de}    

\date{}

\begin{document}

\maketitle

\begin{abstract}
\small{ \noindent 
We establish an exponential stabilization result for linear %first-order 
port-Hamiltonian systems of first order with quite general, not necessarily continuous, energy densities. In fact, we have only to require the energy density of the system to be of bounded variation. In particular, and in contrast to the previously known stabilization results, %for such systems
our result applies to vibrating strings or beams with jumps in their mass density and modulus of elasticity.   
}
\end{abstract}

{ \small \noindent 
%\emph{AMS Subject Classification (2010):} 
%\\
}

\section{Introduction}

%1. Was machen wir in dieser Arbeit in einem Satz?
In this paper, we are concerned with the stabilization of linear first-order port-Hamiltonian systems with discontinuous energy densities  on a bounded interval $[a,b]$. %of first order on a bounded interval $[a,b]$. 
%
%
%2. Was betrachten wir und was wollen wir? (Ausgangslage und Ziel)
%
%2.1 Was ist ein lineares port-hamiltonsches System 1-ter Ordnung?
Such a system evolves according to a partial differential equation of the form %is described by states $x(t,\cdot): [a,b] \to \K^m$ which evolve according to the partial differential equation 
%A linear first-order port-Hamiltonian system on a bounded interval $[a,b]$ is a system whose states $x(t,\cdot): [a,b] \to \K^m$ evolve according to the pde ... and whose energy ...
\begin{align} \label{eq:pde-intro}
\partial_t x(t,\zeta) = P_1 \partial_{\zeta} \big( \mathcal{H}(\zeta) x(t,\zeta) \big) + P_0 \mathcal{H}(\zeta) x(t,\zeta)
\end{align}
for the states $x(t,\cdot): [a,b] \to \K^m$, and the energy of such a system at time $t$ is given by an integral of the form
\begin{align*} %\label{eq:energy-intro}
E(x(t,\cdot)) = \frac{1}{2} \int_a^b x(t,\zeta)^* \mathcal{H}(\zeta) x(t,\zeta) \d \zeta.
\end{align*}  
In these equations, $\mathcal{H}$ is the %so-called 
energy density of the system, that is, a suitable measurable function from $[a,b]$ to $\K^{m\times m}$, and $P_0, P_1$ are %suitable 
matrices in $\K^{m\times m}$ with suitable symmetry %and invertibility 
properties.
%
%2.2 Was wollen wir erreichen? Was bedeutet lineare boundary control? 
We want to stabilize %aim at stabilizing
such systems by linear boundary control and therefore we complement~\eqref{eq:pde-intro} by %add to~\eqref{eq:pde-intro}
the linear boundary condition
\begin{align}
0 = W_{B,1} \begin{pmatrix} \mathcal{H}(b)x(t,b) \\ \mathcal{H}(a)x(t,a) \end{pmatrix}
\end{align}
and the linear boundary input and output
\begin{align}
u(t) = W_{B,2} \begin{pmatrix} \mathcal{H}(b)x(t,b) \\ \mathcal{H}(a)x(t,a) \end{pmatrix}
\qquad \text{and} \qquad
y(t) = W_C \begin{pmatrix} \mathcal{H}(b)x(t,b) \\ \mathcal{H}(a)x(t,a) \end{pmatrix},
\end{align}
where $W_{B,1} \in \K^{(m-k)\times 2m}$ and $W_{B,2}, W_C \in \K^{k \times 2m}$ and $k \in \{ 1, \dots, m\}$. 
%
%2.3 Was betrachten wir, abstrakt gesprochen, fuer Systems?
So, in abstract terms, %and more concise terms
we consider %systems described by 
a linear evolution equation 
\begin{align} \label{eq:evol-eq-intro}
\dot{x} = \mathcal{A}x = P_1 \partial_{\zeta}(\mathcal{H}x) + P_0 \mathcal{H}x
\end{align} 
in the state space $X := L^2([a,b],\K^m)$ with additional linear boundary input and output conditions
\begin{align} \label{eq:bdry-in/output-cond-intro}
u(t) = \mathcal{B}x(t) \qquad \text{and} \qquad y(t) = \mathcal{C}x(t),
\end{align}
where the linear differential operator $\mathcal{A}: D(\mathcal{A}) \subset X \to X$ is defined by the right-hand side of~\eqref{eq:evol-eq-intro} with domain
\begin{align*}
D(\mathcal{A}) := \bigg\{ x \in X: \mathcal{H}x \in W^{1,2}((a,b),\K^m) \text{ and } W_{B,1} \begin{pmatrix} (\mathcal{H}x)(b)  \\ (\mathcal{H}x)(a) \end{pmatrix} = 0 \bigg\}
\end{align*}
and where the linear boundary in- and output operators $\mathcal{B}, \mathcal{C}: D(\mathcal{A}) \subset X \to \K^k$ are defined by
\begin{align*}
\mathcal{B}x := W_{B,2} \begin{pmatrix} (\mathcal{H}x)(b)  \\ (\mathcal{H}x)(a) \end{pmatrix}
\qquad \text{and} \qquad
\mathcal{C}x := W_C \begin{pmatrix} (\mathcal{H}x)(b)  \\ (\mathcal{H}x)(a) \end{pmatrix}.
\end{align*}
\smallskip

%
%
%3. Was zeigen wir in dieser Arbeit (und unter welchen Voraussetzungen)? Anwendungen?

%3.1 Was zeigen wir?
What we show in this paper is that the input-output system~\eqref{eq:evol-eq-intro}, \eqref{eq:bdry-in/output-cond-intro} can be  exponentially stabilized by means of the negative output-feedback law
\begin{align} \label{eq:output-feedback-intro}
u(t) = -\mu y(t)
\end{align}
with an arbitrary $\mu > 0$, %being an arbitrary positive (output) amplification factor
that is, the system~\eqref{eq:evol-eq-intro}, \eqref{eq:bdry-in/output-cond-intro} with the additional feedback condition~\eqref{eq:output-feedback-intro} is an exponentially stable linear system. 
%
%3.2 Unter welchen Voraussetzungen zeigen wir dieses Stabilitaetsaussage?
We achieve %obtain/show 
this exponential stability of the closed-loop system~\eqref{eq:evol-eq-intro}, \eqref{eq:bdry-in/output-cond-intro}, \eqref{eq:output-feedback-intro} 
under the assumption that the energy density $\zeta \mapsto \mathcal{H}(\zeta)$ is of bounded variation and that the open-loop system~\eqref{eq:evol-eq-intro}, \eqref{eq:bdry-in/output-cond-intro} satisfies two additional natural conditions, namely (i) impedance-passivity and (ii) domination of the state value at one of the boundary points ($a$ or $b$) by the input and output. 
%
%3.3 Worauf wenden wir den Stabilisierungssatz an?
We apply this stability result to vibrating strings and beams (modelled according to Timoshenko). 
\smallskip

%
%
%4. Wodurch unterscheidet sich unser Stabilisierungssatz von den bisher bekannten?
Since the energy density in our result is only required to be of bounded variation, we can treat strings and beams with jumps in their material characteristics like mass density or modulus of elasticity. %whose material characteristics like mass density or modulus of elasticity have jumps. 
With the previously known stabilization results, by contrast, such situations %applications ... were not possible
with jumps in the mass density and the modulus of elasticity cannot be dealt with. %covered. 
Indeed, the stability results from~\cite{ViZwLeMa09}, \cite{JaZw}, \cite{Au} %for general linear first-order port-Hamiltonian systems 
are restricted to port-Hamiltonian systems with continuously differentiable or Lipschitz continuous energy densities, and the stability result from~\cite{CoZu95} is restricted to vibrating strings with constant modulus of elasticity (while allowing bounded variation regularity for the mass density). %(while allowing mass densities of bounded variation). 
\smallskip

%
%
%5. Schreib- und Sprechweisen
In the entire paper, we will use the following notations. As usual, $\mathbb{K}$ stands %either 
for the field $\R$ of real or the field $\C$ of complex numbers, $\R^+_0 := [0,\infty)$ denotes the set of non-negative reals, and $|\cdot|$ denotes the standard norm on $\K^m$ for any $m \in \N$. 
Also, $L^p(S,\K^m)$ and $W^{k,p}((a,b),\K^m)$ for $p \in [1,\infty) \cup \{\infty\}$ (integrability index) and $k \in \N$ (differentiability index) are the usual Lebesgue and Sobolev spaces, respectively, and $\norm{\cdot}_p$ and $\scprd{\cdot,\cdot \cdot}_2$ will denote the standard norm and scalar product of $L^p(S,\K^m)$ and $L^2(S,\K^m)$, respectively.
$AC([a,b],\K^m)$ denotes the set of absolutely continuous functions from $[a,b]$ to $\K^m$.
And finally, for $J = [a,b]$ or $J = \R$, 
\begin{align*}
BV(J,\K^m) := \big\{ \text{functions } f: J \to \K^m \text{ with } \operatorname{Var}(f) < \infty \big\}
\end{align*} 
denotes the set of functions of bounded variation from $J$ to $\K^m$, where 
\begin{align*}
\operatorname{Var}(f) := \bigg\{ \sum_{l=1}^{L} |f(t_l)-f(t_{l-1})|: (t_l)_{l \in \{0,\dots,L\} } \text{ a partition of } [a,b] \bigg\}
\end{align*}
in the case $J = [a,b]$ and where $\operatorname{Var}(f) := \sup \{ \operatorname{Var}(f|_{[a,b]}): a < b \}$ in the case $J = \R$.

\section{Some technical preliminaries}
%\section{Some technical preliminaries about functions with values in function spaces} 
%about the measurable (concrete) representation of (abstract) function-space valued functions
%about the measurability of concrete represantatives of abstract (function-space-valued) functions

In this section, we record %provide 
some technical preliminaries about measurable representations of  functions with values in function spaces. 
In essence, the following lemma can be found in~\cite{HillePhillips} (Section~3.4, paragraph about spaces of class $L$), but the importance of choosing the right representatives -- which is demonstrated by our  example below -- is ignored there. %is not pointed out there. %ignored there. %overlooked there. 
%
%In the following, a function is called ... 
We recall from~\cite{AmEs} (Chapter~X.1) that a function $f: S \to X$ between a measurable set $S \in \mathcal{L}_{\R^d}$ (Lebesgue $\sigma$-algebra on $\R^d$) and a Banach space $X$ is called \emph{$\lambda$-measurable} iff there is a sequence of integrable simple functions $f_n: S \to X$  converging to $f$ $\lambda$-almost everywhere, where $\lambda$ is the Lebesgue measure on $\mathcal{L}_{\R^d}$. In case $X$ is separable, %and a fortiori if it is finite-dimensional
$\lambda$-measurability coincides with (plain) $\mathcal{L}_{\R^d}$-$\mathcal{B}_{X}$-measurability by Pettis' theorem %(Theorem~X.1.4 of~\cite{AmES})  
($\mathcal{B}_X$ being the Borel $\sigma$-algebra of $X$).
A function $f: S \to X$ as above will be called \emph{$p$-integrable} for a $p \in [1,\infty)$ iff it is $\lambda$-measurable and 
\begin{align*}
\int_S \norm{f(s)}_X^p \d s := \int_S \norm{f(s)}_X^p \d \lambda(s) < \infty.
\end{align*}

%We will call a function $f: S \to X$ between a set $S \in \mathcal{L}_{\R^d}$ and a Banach space $X$ $p$-integrable for a $p \in [1,\infty)$ iff it is $\lambda$-measurable and 
%\begin{align*}
%\int_S \norm{f(s)}_X^p \d s := \int_S \norm{f(s)}_X^p \d \lambda(s) < \infty,
%\end{align*}
%where $\mathcal{L}_{\R^d}$ is the Lebesgue $\sigma$-algebra on $\R^d$, $\lambda: \mathcal{L}_{\R^d} \to [0,\infty) \cup \{\infty\}$ is the Lebesgue measure on $\R^d$, and $\norm{\cdot}_X$ is the norm of $X$. See~\cite{AmEs} for the definition of $\lambda$-measurability and recall that in the case of separable spaces $X$ $\lambda$-measurability coincides with $\mathcal{L}_{\R^d}-\mathcal{B}_{X}$-measurability ($\mathcal{B}_X$ the Borel $\sigma$-algebra of $X$). 

\begin{lm} \label{lm:mb representation of abstract functions}
Suppose $f: J \to X$ is a $p$-integrable function from a bounded  interval $J \subset \R$ to the space $X := L^p(Z,\K^m)$, where $p \in [1,\infty)$ and $Z \in \mathcal{L}_{\R}$. Then 
\begin{itemize}
\item[(i)] for every $s \in J$ there is a representative $\ul{f}(s): Z \to \K^m$ of $f(s)$ such that the %concrete 
function 
\begin{align} \label{eq:mb representation, concrete function}
J\times Z \ni (s,\zeta) \mapsto \ul{f}(s)(\zeta) \in \K^m
\end{align} 
is measurable and
\item[(ii)] for every choice of representatives $\ul{f}(s)$ as in~(i) %for every choice of representatives $\ul{f}(s)$ such that~\eqref{eq:mb representation, concrete function} is measurable 
the function $J \ni s \mapsto \ul{f}(s)(\zeta)$ is integrable for a.e.~$\zeta \in Z$ and 
\begin{align}  \label{eq:mb representation, integral of concrete function}
\zeta \mapsto \int_J \ul{f}(s)(\zeta) \d s 
\end{align}
is a representative of the element $\int_J f(s) \d s \in X$.
\end{itemize} 
\end{lm}

\begin{proof}
We strictly distinguish between functions and equivalence classes of functions in this proof and, as usual, we use square brackets to denote equivalence classes.
\smallskip

(i) Since $f: J \to X$ is $p$-integrable, there exist integrable simple functions $f_n: J \to X$ such that
$f_n(s) \longrightarrow f(s)$ for a.e.~$s \in J$ and we can also assume that $$\norm{f_n(s)}_X \le 2 \norm{f(s)}_X$$ for all $s \in J$ and $n \in \N$ (if this bound does not hold for the initial choice of simple functions $f_n^0$, just multiply them by the characteristic function of the (measurable!) set $\{s \in J: \norm{f_n^0(s)}_X \le 2 \norm{f(s)}_X \}$). So, by the theorem of dominated convergence,   
\begin{align} \label{eq:mb representation, 1}
\int_J \norm{f_n(s)-f(s)}_X^p \d s \longrightarrow 0 \qquad (n \to \infty).
\end{align}
Since the $f_n$ are $\lambda$-measurable simple functions, they are of the form
\begin{align} \label{eq:mb representation, canonical form of f-n} 
f_n(s) = \sum_{k=1}^{m_n} \alpha_{n k} \chi_{E_{n k}}(s) \qquad (s \in J)
\end{align}
for certain $\alpha_{n k} \in X = L^p(Z,\K^m)$ and $E_{n k} \in \mathcal{L}_{\R}$. Choosing representatives $\ul{\alpha}_{n k}: Z \to \K^m$ of $\alpha_{n k}$ and defining $\ul{\phi}_n: J \times Z \to \K^m$ by 
\begin{align*}
\ul{\phi}_n(s,\zeta) := \sum_{k=1}^{m_n} \ul{\alpha}_{n k}(\zeta) \chi_{E_{n k}}(s) \qquad ((s,\zeta) \in J \times Z),
\end{align*}
we see that $\ul{\phi}_n$ is measurable for every $n \in \N$ so that by Tonelli's theorem (Theorem~X.6.7 of~\cite{AmEs}) and~\eqref{eq:mb representation, 1} we have
\begin{align*}
\int_{J \times Z} |\ul{\phi}_n(s,\zeta) - \ul{\phi}_m(s,\zeta)|^p \d (s,\zeta) = \int_J \norm{f_n(s)-f_m(s)}_X^p \d s \longrightarrow 0 \qquad (m, n \to \infty).
\end{align*}
So, by the completeness of $L^p(J\times Z, \K^m)$, there is a $p$-integrable function $\ul{\phi}: J \times Z \to \K^m$ such that
\begin{align} \label{eq:mb representation, 2}
\int_{J \times Z} |\ul{\phi}_n(s,\zeta) - \ul{\phi}(s,\zeta)|^p \d (s,\zeta) \longrightarrow 0 \qquad (n \to \infty).
\end{align}
We have by Tonelli's theorem that $[\ul{\phi}(s,\cdot)]$, $[\ul{\phi}_n(s,\cdot) - \ul{\phi}(s,\cdot)]$ belong to $L^p(Z,\K^m)$ for a.a.~$s \in J$ (with exceptional sets $N_0$ and $N_n$ respectively) and that 
$$J\setminus N' \ni s \mapsto \norm{[\ul{\phi}(s,\cdot)]}_X^p, \big\| [\ul{\phi}_n(s,\cdot) - \ul{\phi}(s,\cdot)] \big\|_X^p $$
are measurable for all $n \in \N$, where $N' := \bigcup_{n=0}^{\infty} N_n$.  In view of~\eqref{eq:mb representation, 2} it now follows that
\begin{align} \label{eq:mb representation, 3}
\int_J \norm{f_n(s) - [\ul{\phi}(s,\cdot)]}_X^p \d s = \int_J \int_Z |\ul{\phi}_n(s,\zeta) - \ul{\phi}(s,\zeta)|^p \d \zeta \d s \longrightarrow 0 \qquad (n \to \infty).
\end{align} 
Combining~\eqref{eq:mb representation, 1} and~\eqref{eq:mb representation, 3} we see that 
\begin{align} \label{eq:mb representation, 4}
f(s) = [\ul{\phi}(s,\cdot)]
\end{align}
for a.a.~$s \in J$ (with an exceptional set denoted by $N''$). We now define $\ul{f}(s): Z \to \K^m$ by
\begin{align*}
\ul{f}(s)(\zeta) 
:=
\begin{cases}
\ul{\phi}(s,\zeta), \quad (s,\zeta) \in (J \setminus N'') \times Z \\
\ul{f}_0(s)(\zeta), \quad (s,\zeta) \in N'' \times Z
\end{cases} 
\end{align*}
where $\ul{f}_0(s)$ for $s \in N''$ is an arbitrary representative of $f(s)$. It then follows by~\eqref{eq:mb representation, 4}  that $\ul{f}(s)$ is a representative  of $f(s)$ for every $s \in J$ and by the measurability of $\ul{\phi}$ and $\lambda(N''\times Z) = 0$ %and the completeness of $\mathcal{L}_{J \times Z}$
it follows that $J \times Z \ni (s,\zeta) \mapsto \ul{f}(s)(\zeta)$ is measurable, as desired.
\smallskip

(ii) Choose and fix for every $s \in J$ a representative $\ul{f}(s)$ of $f(s)$ such that $(s,\zeta) \mapsto \ul{f}(s)(\zeta)$ is measurable %and hence the positive and negative parts $(s,\zeta) \mapsto \big( \ul{f}(s)(\zeta) \big)_{\pm}$ are measurable as well
(which is possible by part~(i)). 
%
%Schritt 1
It follows by Tonelli's theorem %(Theorem~X.6.7 of~\cite{AmEs} or Theorem 8.12 in conjunction with Theorem 8.11 of~\cite{Rudin})
that %$s \mapsto \big( \ul{f}(s)(\zeta) \big)_{\pm}$ is measurable for a.e.~$\zeta \in Z$ (exceptional set $N_+$, $N_-$) and therefore that 
$s \mapsto \ul{f}(s)(\zeta)$ %= \big( \ul{f}(s)(\zeta) \big)_+ - \big( \ul{f}(s)(\zeta) \big)_-
is measurable for a.e.~$\zeta \in Z$ and that $\zeta \mapsto \int_{J} |\ul{f}(s)(\zeta)| \d s \in [0,\infty) \cup \{\infty\}$ is measurable as well. Also,  
\begin{align*}
\int_Z \bigg( \int_J |\ul{f}(s)(\zeta)| \d s \bigg)^p \d \zeta 
&\le \int_Z \lambda(J)^{p/q} \int_J |\ul{f}(s)(\zeta)|^p \d s \d \zeta
= \lambda(J)^{p/q} \int_J \norm{f(s)}_X^p \d s \\
&< \infty
\end{align*}
by the boundedness of $J$ and the $p$-integrability of $f$ (where $q \in (1,\infty]$ is the dual exponent of $p \in [1,\infty)$, of course). Consequently, 
\begin{align*}
\int_J |\ul{f}(s)(\zeta)| \d s < \infty
\end{align*}
for a.e.~$\zeta \in Z$ and thus the function $J \ni s \mapsto \ul{f}(s)(\zeta)$ is integrable for a.e.~$\zeta \in Z$. 
%
%Schritt 2
What we have to show now is that
\begin{align} \label{eq:mb representation, 5}
\zeta \mapsto \int_J \ul{f}(s)(\zeta) \d s
\end{align}
is a representative of $F := \int_J f(s) \d s \in X$ (where the existence of this integral in $X$ follows by means of H\"older's inequality from the $p$-integrability of $f$ and the boundedness of $J$). 
In order to do so, we show that for any %arbitrary 
given representative $\ul{F}$ of $F$ one has 
\begin{align} \label{eq:mb representation, 5-1}
\int_J \ul{f}(s)(\zeta) \d s = \ul{F}(\zeta)
\end{align}
for a.e.~$\zeta \in Z$. 
Choose integrable simple functions $f_n: J \to X$ such that  
\begin{align} \label{eq:mb representation, 6}
\int_J \norm{f_n(s)-f(s)}_X^p \d s \longrightarrow 0 \qquad (n \to \infty)
\end{align}
(see the beginning of the proof of part~(i)) and write $F_n := \int_J f_n(s)\d s$. Also, for every $s \in J$ and $n \in \N$ choose a representative $\ul{f}_n(s)$ of $f_n(s)$ and $\ul{F}_n$ of $F_n$ by choosing representatives of the values $\alpha_{n k} \in X$ of $f_n$, see~\eqref{eq:mb representation, canonical form of f-n}.  
%
%Schritt 2.1
Clearly,  
\begin{align} \label{eq:mb representation, 7}
\int_J \ul{f}_n(s)(\zeta) \d \zeta = \ul{F}_n(\zeta)
\end{align}
for a.e.~$\zeta \in Z$ and every $n \in \N$. 
%
%Schritt 2.2
In view of~\eqref{eq:mb representation, 6} it further follows that
\begin{align} \label{eq:mb representation, 8}
\int_Z |\ul{F}_n(\zeta) - \ul{F}(\zeta)|^p \d \zeta = \norm{F_n - F}_X^p 
\le \lambda(J)^{p/q} \int_J \norm{f_n(s)-f(s)}_X^p \d s
\longrightarrow 0 %\qquad (n \to \infty)
\end{align} 
as $n \to \infty$ and that 
\begin{align} \label{eq:mb representation, 9}
\int_Z \bigg| \int_J \ul{f}_n(s)(\zeta) \d s - \int_J \ul{f}(s)(\zeta) \d s \bigg|^p \d \zeta 
&\le \int_Z \lambda(J)^{p/q} \int_J |\ul{f}_n(s)(\zeta) - \ul{f}(s)(\zeta)|^p \d s \d \zeta \notag \\
&= \lambda(J)^{p/q} \int_J \norm{f_n(s)-f(s)}_X^p \d s
\longrightarrow 0 %\qquad (n \to \infty)
\end{align}
as $n \to \infty$. So by~\eqref{eq:mb representation, 8} and~\eqref{eq:mb representation, 9} there is a subsequence $(n_k)$ such that
\begin{align} \label{eq:mb representation, 10}
\ul{F}_{n_k}(\zeta) \longrightarrow \ul{F}(\zeta)  \qquad (k \to \infty)
\end{align}
for a.e.~$\zeta \in Z$ and such that
\begin{align} \label{eq:mb representation, 11}
\int_J \ul{f}_{n_k}(s)(\zeta) \d s \longrightarrow \int_J \ul{f}(s)(\zeta) \d s  \qquad (k \to \infty)
\end{align}
for a.e.~$\zeta \in Z$. Combining now~\eqref{eq:mb representation, 7} with \eqref{eq:mb representation, 10} and \eqref{eq:mb representation, 11}, we obtain the desired equality~\eqref{eq:mb representation, 5-1} for almost every $\zeta \in Z$. 
\end{proof}

\begin{cor} \label{cor:mb representation of abstract functions}
Suppose $f: J \to X$ is a continuous function from a compact interval $J \subset \R$ to the space $X := L^p(Z,\K^m)$, where $p \in [1,\infty)$ and $Z \in \mathcal{L}_{\R}$. Then the conclusions of the previous lemma hold true. 
%Then for every $s \in J$ there is a representative $\ul{f}(s): Z \to \K^m$ of $f(s)$ such that the concrete function~\eqref{lm:mb representation, concrete function} 
%%\begin{align} \label{lm:mb representation, assertion}
%%J\times Z \ni (s,\zeta) \mapsto (\ul{f}(s))(\zeta) \in \K^m
%%\end{align} 
%is $\lambda$-measurable. 
\end{cor}

\begin{proof}
Since $f$ is continuous, it is $\mathcal{L}_{\R}$-$\mathcal{B}_X$-measurable and separably valued. So, $f$ is $\lambda$-measurable by Pettis' theorem (Theorem~X.1.4 of~\cite{AmEs}). Since moreover $J$ is compact, $f$ is $p$-integrable and thus the assertion follows by the previous lemma. 
\end{proof}

In view of the previous lemma, the question arises whether (i) for every choice of representatives $\ul{f}(s)$ of $f(s)$, the function~\eqref{eq:mb representation, concrete function} is measurable and whether (ii) for every choice of representatives $\ul{f}(s)$ of $f(s)$ such that $J \ni s \mapsto \ul{f}(s)(\zeta)$ is integrable for a.e.~$\zeta \in Z$, the function~\eqref{eq:mb representation, integral of concrete function} is a representative of the element $\int_J f(s) \d s \in X$. 
As the following example shows, the answers to both questions are  negative. %no. %that is not the case. 

\begin{ex} \label{ex:mb representation of abstract functions}
(i) Set $J, Z := [0,1]$ and choose a subset $E$ of $J \times Z$ such that $E$ is not Lebesgue-measurable and such that each line in $\R^2$ intersects $E$ in at most $2$ points. Such a set has been shown to exist by Sierpi\'{n}ski in~\cite{Si20a} using the axiom of choice. %(but not the continuum hypothesis). 
Also, let $f: J \to X := L^2(Z,\R)$ and $\ul{f}(s): J \to \R$ be defined by
\begin{align}
f(s) := 0 \qquad (s \in J) \qquad \text{and} \qquad \ul{f}(s)(\zeta) := \chi_{E}(s,\zeta) \qquad ((s,\zeta) \in J\times Z).
\end{align}
Since the section $E_s := \{ \zeta \in Z: (s,\zeta) \in E\}$ has at most $2$ elements for every $s \in J$, the function $\ul{f}(s)$ is a representative of $f(s)$ for every $s \in J$ but the function~\eqref{eq:mb representation, concrete function} is not measurable because $E \notin \mathcal{L}_{\R^2}$.
\smallskip

(ii) Set $J, Z := [0,1]$ and choose a subset $E$ of $J \times Z$ such that the section $E_s := \{ \zeta \in Z: (s,\zeta) \in E\}$ is countable for every $s \in J$ and such that the section $E^{\zeta} := \{s \in J: (s,\zeta) \in E\}$ is co-countable for every $\zeta \in Z$. Such a set $E$ was shown to exist by Sierpi\'{n}ski in~\cite{Si20b} assuming that the continuum hypothesis is true (which is not needed for~\cite{Si20a}). (See also Example 8.9~(c) in~\cite{Rudin:real-complex} and Exercise~2.47 and Section~2.8 of~\cite{Folland:real}.)
Also, let $f: J \to X := L^2(Z,\R)$ and $\ul{f}(s): J \to \R$ be defined by
\begin{align}
f(s) := 0 \qquad (s \in J) \qquad \text{and} \qquad \ul{f}(s)(\zeta) := \chi_{E}(s,\zeta) \qquad ((s,\zeta) \in J\times Z).
\end{align}
Since the section $E_s$ is countable for every $s \in J$, the function $\ul{f}(s) = \chi_{E_s}$ is a representative of $f(s)$ for every $s \in J$, and since $E^{\zeta}$ is co-countable, $J \ni s \mapsto \ul{f}(s)(\zeta) = \chi_{E^{\zeta}}(s)$ is integrable for every $\zeta \in Z$ but, 
as
\begin{align}
\int_J \ul{f}(s)(\zeta) \d s = 1 \ne 0 \qquad (\zeta \in Z),
\end{align}
the function~\eqref{eq:mb representation, integral of concrete function} is not a representative of $0 = \int_J f(s) \d s \in X$. 
%because
%\begin{align}
%\int_J \ul{f}(s)(\zeta) \d s = 1 \ne 0 \qquad (\zeta \in Z).
%\end{align}
$\blacktriangleleft$
\end{ex}

\section{Stability results}
%\section{Stabilization result}

In this section, we establish the main stability results of this paper and to do so we need some preparations. %preparatory lemmas. 
%
%DEFINITION VON PORT-HAMILTONSCHEN ABBILDUNGEN UND SYSTEMEN (SOWIE VON ENERGIEDICHTEN UND ENERGIENORMEN UND SKALARPRODUKTEN)
We will call a matrix-valued function $[a,b] \ni \zeta \mapsto \mathcal{H}(\zeta) \in \K^{m\times m}$ on some compact interval $[a,b]$ an \emph{energy density} iff it is measurable, $\mathcal{H}(\zeta)$ is self-adjoint for almost all $\zeta \in [a,b]$, and there are constants $\ul{m}, \ol{m} \in (0,\infty)$ such that 
\begin{align} \label{eq:en-density, lower and upper bound}
\ul{m} \le \mathcal{H}(\zeta) \le \ol{m} 
\end{align}
for almost all $\zeta \in [a,b]$. 
Also, for a given energy density $\mathcal{H}$, a linear operator $A: D(A) \subset X \to X$ is called a \emph{first-order port-Hamiltonian operator with energy density $\mathcal{H}$} iff the domain 
\begin{align*}
D(A) \subset \{ x \in X: \mathcal{H}x \in W^{1,2}((a,b),\K^m) \}
\end{align*}
is a dense subspace of $X := L^2([a,b],\K^m)$ and if $A$
%its domain is dense in $X := L^2(Z,\K^m)$ and satisfies 
%\begin{align*}
%D(A) \subset \{ x \in X: \mathcal{H}x \in W^{1,2}(Z,\K^m) \}
%\end{align*}
%and if the operator %acts as follows: 
is of the form
\begin{align}
Ax = P_1 \partial_{\zeta} (\mathcal{H}x) + P_0 \mathcal{H}x 
\qquad (x \in D(A))
\end{align} 
%for $x \in D(A)$ with some 
for some invertible self-adjoint matrix $P_1 = P_1^* \in \K^{m\times m}$ and some skew-adjoint matrix $P_0 = -P_0^* \in \K^{m\times m}$. 
%
%Correspondingly, an evolution equation
An evolution equation $\dot{x} = Ax$ 
%\begin{align*}
%\dot{x} = Ax
%\end{align*}
with $A$ being a first-order port-Hamiltonian operator is called a \emph{first-order port-Hamiltonian system}.
Additionally, the scalar product $\scprd{\cdot, \cdot\cdot}_X$  defined by
\begin{align} \label{eq:en-scprd-def}
\scprd{x,y}_X := \frac{1}{2} \int_a^b x(\zeta)^* \mathcal{H}(\zeta) y(\zeta) \d \zeta
\end{align}
is called the \emph{$\mathcal{H}$-energy scalar product} and the corresponding norm $\norm{\cdot}_X$ is called the \emph{$\mathcal{H}$-energy norm}. 
%
%ANMERKUNGEN/EINFACHE SCHLUSSFOLGERUNGEN
In view of~\eqref{eq:en-density, lower and upper bound} it is clear that the $\mathcal{H}$-energy norm is equivalent to the standard norm of $L^2(Z,\K^m)$. 
In view of the continuous embedding of $W^{1,2}((a,b),\K^m)$ in $C([a,b],\K^m)$ it is also clear that for $x \in D(A)$ the vector
\begin{align*}
(\mathcal{H}x)|_{\partial} := \begin{pmatrix} (\mathcal{H}x)(b) \\ (\mathcal{H}x)(a) \end{pmatrix} \in \K^{2m}
\end{align*}
of stacked boundary values is well-defined. 
%
%VEREINBARUNG: IDENTIFIZIERUNG VON \mathcal{H} MIT DEM ZUGEHOERIGEN MULT.OP.
As usual, we do not distinguish here and in the following between $\mathcal{H}$ and the multiplication operator $M_{\mathcal{H}}$ associated with $\mathcal{H}$, that is, we will always write $\mathcal{H}x$ for $M_{\mathcal{H}} x$. Similarly, $\mathcal{H}^{-1}$ will stand for $\zeta \mapsto \mathcal{H}(\zeta)^{-1}$ as well as for the corresponding multiplication operator. 
\smallskip

%VORBEREITENDES LEMMA 1 (CHAR KHGR-ERZ-EIGENSCHAFT FUER PORT-HAM OP)

%We begin by recalling ...
As a first preparatory lemma, we recall from~\cite{JaMoZw15} (Theorem~1.1) the following characterization of when a port-Hamiltonian operator generates a contraction semigroup.

\begin{lm} \label{lm:char-contr-sgr-gen-property}
Suppose $A: D(A) \subset X \to X$ is a first-order port-Hamiltonian operator with energy density $\mathcal{H}: [a,b] \to\K^{m\times m}$, where $X := L^2([a,b],\K^m)$ is endowed with the $\mathcal{H}$-energy norm $\norm{\cdot}_X$. Suppose further that the domain of $A$ incorporates $m$ linear boundary conditions, that is, it is of the form
\begin{align} \label{eq:domain-with-m-lin-bdry-cond}
D(A) = \big\{ x \in X: \mathcal{H}x \in W^{1,2}((a,b),\K^m) \text{ and } W(\mathcal{H}x)|_{\partial} = 0 \big\}
\end{align}
for some matrix $W \in \K^{m\times 2m}$. Then $A$ generates a contraction semigroup on $X$ if and only if $A$ is dissipative in $X$, that is, 
\begin{align*}
\Re \scprd{x,Ax}_X \le 0 \qquad (x \in D(A)). 
\end{align*}
In that case, the boundary matrix $W$ automatically has full rank $m$. 
\end{lm}

%VORBEREITENDES LEMMA 2 (DB-LEMMA FUER SIDEWAYS ENERGIES FUNKTIONEN)

As a second preparatory lemma, we show the following differentiability result for certain sideways energy functions along classical solutions of port-Hamiltonian systems with absolutely continuous energy densities. (A classical solution of a such a system~\eqref{eq:xdot=Ax} is a continuously differentiable map $x: J \to X$ on some interval $J \subset \R^+_0$ such that for all $t \in J$ one has $x(t) \in D(A)$ and $\dot{x}(t) = Ax(t)$.) %We recall here 
In~\cite{JaZw}, \cite{Au} such a differentiability result -- for the special case of continuously differentiable or Lipschitz continuous energy densities --  is used implicitly as well, but no proofs are given there. As we will see, the proof requires quite some work and care. In fact, some of the (formal) computations from~\cite{JaZw}, \cite{Au} will in general become false for careless choices of representatives. See the example below.   

\begin{lm} \label{lm:F db a.e.}
Suppose $A: D(A) \subset X \to X$ is a first-order port-Hamiltonian operator on $X := L^2([a,b],\K^m)$ with energy density $\mathcal{H} \in AC([a,b],\K^{m\times m})$. %, where $X = L^2((a,b),\K^m)$. 
Suppose further that $x: \R^+_0 \to X$ 
is a classical solution of the differential equation 
\begin{align} \label{eq:xdot=Ax}
\dot{x} = Ax
\end{align}
%is continuously differentiable such that
%\begin{align*}
%x(s) \in D(\mathcal{A}) \qquad \text{and} \qquad \dot{x}(s) = \mathcal{A}x(s) 
%\end{align*} 
%for all $s \in \R^+_0$ 
and let $F: [a,b] \to \K$ be the sideways energy defined by %function defined by
\begin{align} \label{eq:def von F}
F(\zeta) := \int_{r(\zeta)}^{t(\zeta)} \ul{x}(s)(\zeta)^* \mathcal{H}(\zeta) \ul{x}(s)(\zeta) \d s, 
\end{align} 
where $r,t \in C^1([a,b],\R^+_0)$ are given functions %$r,t: [a,b] \to \R^+_0$ are given differentiable functions with $\norm{r'}_{Z,\infty}, \norm{t'}_{Z,\infty} < \infty$ 
and where $\ul{x}(s)$ for every $s \in \R^+_0$ is the continuous representative of $x(s)$. 
It then follows that $F$ is absolutely continuous and hence differentiable almost everywhere with derivative given by
\begin{align} \label{eq:abl von F}
F'(\zeta) &= \ul{x}(s)(\zeta)^* \big( t'(\zeta) \mathcal{H}(\zeta) + P_1^{-1} \big) \ul{x}(s)(\zeta) \Big|_{s=t(\zeta)} 
- \ul{x}(s)(\zeta)^* \big( r'(\zeta) \mathcal{H}(\zeta) + P_1^{-1} \big) \ul{x}(s)(\zeta) \Big|_{s=r(\zeta)} \notag \\
&\quad - \int_{r(\zeta)}^{t(\zeta)}  \ul{x}(s)(\zeta)^* \big(  (P_1^{-1} P_0 \mathcal{H}(\zeta))^{*} + \mathcal{H}'(\zeta) + P_1^{-1} P_0 \mathcal{H}(\zeta)  \big) \ul{x}(s)(\zeta) \d s
\end{align}
for almost every $\zeta \in [a,b]$.
Additionally, for $\mathcal{H} \in C^1([a,b],\K^{m\times m})$ the sideways energy $F$ defined above is even continuously differentiable.
\end{lm}

\begin{proof}
We divide the proof into two parts. In part~(i) we prove in five steps  the assertion for $\mathcal{H} \in AC([a,b],\K^{m\times m})$ and in part~(ii) we prove the strengthening for $\mathcal{H} \in C^1([a,b],\K^{m\times m})$. In the entire proof, we abbreviate $Z := [a,b]$ and $Z^{\circ} := (a,b)$. 
\smallskip

%Schritt 1
(i) As a first step, we observe that $x(s) \in W^{1,1}(Z^{\circ},\K^m)$ for every $s \in \R^+_0$ and that $s \mapsto x(s) \in W^{1,1}(Z^{\circ},\K^m)$ is continuous. 
Indeed, since $x$ is a classical solution of~\eqref{eq:xdot=Ax}, we have that
\begin{align} \label{eq:F db, step 1, 1}
\mathcal{H}x(s) \in W^{1,2}(Z^{\circ},\K^m) \qquad (s \in \R^+_0)
\end{align}
%by the definition of $D(\mathcal{A})$ 
and that
%\begin{align}
%s \mapsto \mathcal{H}x(s) \in X = L^2(Z,\K^m) 
%\qquad \text{as well as} \qquad
%s \mapsto \partial_{\zeta}( \mathcal{H}x(s) ) = P_1^{-1} \dot{x}(s) - P_1^{-1} P_0 \mathcal{H}x(s) \in X = L^2(Z,\K^m) 
%\end{align}
%are continuous.
$s \mapsto \mathcal{H}x(s) \in X = L^2(Z,\K^m) $ as well as
\begin{align} \label{eq:F db, step 1, 2}
s \mapsto \partial_{\zeta}( \mathcal{H}x(s) ) = P_1^{-1} \dot{x}(s) - P_1^{-1} P_0 \mathcal{H}x(s) \in X = L^2(Z,\K^m) 
\end{align}
are continuous. So, $s \mapsto \mathcal{H}x(s)$ is continuous as a function with values in $W^{1,2}(Z^{\circ},\K^m)$. Since $W^{1,2}(Z^{\circ},\K^m)$ is continuously embedded in $W^{1,1}(Z^{\circ},\K^m)$, 
\begin{align} \label{eq:F db, step 1, 2+}
s \mapsto \mathcal{H}x(s)\in  W^{1,1}(Z^{\circ},\K^m)
\end{align}
is continuous as well. Since moreover $\mathcal{H}$ belongs to $W^{1,1}(Z^{\circ},\K^{m\times m})$ by assumption, we also have that %$\mathcal{H}^{-1}$ belongs to $W^{1,1}(Z^{\circ},\K^{m\times m})$ as well with
\begin{align} \label{eq:F db, step 1, 3}
\mathcal{H}^{-1} \in W^{1,1}(Z^{\circ},\K^{m\times m})
\qquad \text{with} \qquad
\partial_{\zeta} \mathcal{H}(\zeta)^{-1} = -\mathcal{H}(\zeta)^{-1} \mathcal{H}'(\zeta) \mathcal{H}(\zeta)^{-1}.
\end{align} 
%Since now $\mathcal{H}$ is Lipschitz continuous, $\mathcal{H}$ and hence also $\mathcal{H}^{-1}$ belongs to $W^{1,\infty}
Combining now the continuity of~\eqref{eq:F db, step 1, 2+} and~\eqref{eq:F db, step 1, 3} with the continuity of multiplication 
\begin{align} \label{eq:F db, step 1, mult contin in W^1,1}
W^{1,1}(Z^{\circ},\K) \times W^{1,1}(Z^{\circ},\K) \ni (f,g) \mapsto fg \in W^{1,1}(Z^{\circ},\K) 
\end{align}
in $W^{1,1}(Z^{\circ},\K)$ (Theorem~4.39 in~\cite{AdFo}), %$W^{1,1}(Z,\K)$ Banach algebra
we obtain the assertion of the first step. 
\smallskip

%Schritt 2
As a second step, we observe that for every $r, t \in \R^+_0$ the map $\Phi_{r,t}: Z \to \K$ defined by
\begin{align} \label{eq:F db, step 2}
\Phi_{r,t}(\zeta) := \int_r^t \ul{x}(s)(\zeta)^* \mathcal{H}(\zeta) \ul{x}(s)(\zeta) \d s
\end{align}
is continuous and, in particular, integrable. In this equation, $\ul{x}(s)$ for every $s \in \R^+_0$ is the continuous representative of $x(s) \in W^{1,1}(Z^{\circ},\K^m)$ (first step!). 
Since $s \mapsto  x(s) \in W^{1,1}(Z^{\circ},\K^m)$ is continuous by  the first step, it follows by the continuous embedding of $W^{1,1}(Z^{\circ},\K^m)$ in $C(Z, \K^m)$ that
\begin{align} \label{eq:F db, step 2, 1}
(s,\zeta) \mapsto \ul{x}(s)(\zeta)
\end{align}
is continuous. %and hence locally bounded. 
And therefore, $\Phi_{r,t}$ is continuous as well. 
\smallskip

%Schritt 3
As a third step, we show that for every $r, t \in \R^+_0$ the map $\Phi_{r,t}$ is weakly differentiable with integrable weak derivative given by
\begin{align} \label{eq:F db, step 3}
\partial_{\zeta} \Phi_{r,t}(\zeta) 
&= 
\ul{x}(s)(\zeta)^* P_1^{-1}  \ul{x}(s)(\zeta) \Big|_{s=r}^{s=t}  \notag \\
&\qquad - \int_{r}^{t}  \ul{x}(s)(\zeta)^* \big(  (P_1^{-1} P_0 \mathcal{H}(\zeta))^{*} + \mathcal{H}'(\zeta) + P_1^{-1} P_0 \mathcal{H}(\zeta)  \big) \ul{x}(s)(\zeta) \d s
\end{align}
for almost all $\zeta \in Z$. 
So let $r, t \in \R^+_0$ be fixed with $r \le t$ and set $J := [r,t]$. Combining the continuity of~\eqref{eq:F db, step 1, 2+} and~\eqref{eq:F db, step 1, 3} with the continuity of~\eqref{eq:F db, step 1, mult contin in W^1,1}, we see that 
\begin{align} \label{eq:F db, def psi}
s \mapsto \psi(s) := x(s)^* \mathcal{H}x(s) = (\mathcal{H}x(s))^* \cdot \mathcal{H}^{-1} \cdot  \mathcal{H}x(s) \in W^{1,1}(Z^{\circ},\K)
\end{align}
is continuous. 
%belongs to $W^{1,1}(Z^{\circ},\K)$ for every $s \in \R^+_0$ and that $s \mapsto \psi(s) \in W^{1,1}(Z^{\circ},\K)$ is continuous. 
%Since $\dot{x}(s) = \mathcal{A} x(s) = ...$ 
With~\eqref{eq:F db, step 1, 2} and~\eqref{eq:F db, step 1, 3} it further follows that 
\begin{align*}
\partial_{\zeta} \psi(s) = \dot{x}(s)^* P_1^{-1} x(s) + x(s)^* P_1^{-1} \dot{x}(s) - x(s)^* \big( (P_1^{-1} P_0 \mathcal{H})^* + \mathcal{H}' + P_1^{-1} P_0 \mathcal{H} \big) x(s)
\end{align*}
for all $s \in \R^+_0$. 
Choose now for every $s \in \R^+_0$ a representative $\ul{v}(s)$ of $v(s) := \dot{x}(s)$ such that $(s,\zeta) \mapsto \ul{v}(s)(\zeta)$ is measurable (Corollary~\ref{cor:mb representation of abstract functions}!) and define $\ul{\psi}(s), \ul{\omega}(s): Z \to \K$ by
\begin{align} 
\ul{\psi}(s)(\zeta) &:= \ul{x}(s)(\zeta)^* \mathcal{H}(\zeta) \ul{x}(s)(\zeta) 
\label{eq:F db, def ul-psi} \\
\ul{\omega}(s)(\zeta)& := \ul{v}(s)(\zeta)^* P_1^{-1} \ul{x}(s)(\zeta) + \ul{x}(s)(\zeta)^* P_1^{-1} \ul{v}(s)(\zeta)   \notag \\
&\qquad - \ul{x}(s)(\zeta)^* \big(  (P_1^{-1} P_0 \mathcal{H}(\zeta))^{*} + \mathcal{H}'(\zeta) + P_1^{-1} P_0 \mathcal{H}(\zeta)  \big) \ul{x}(s)(\zeta) 
\label{eq:F db, def ul-omega}
\end{align}
for all $(s,\zeta) \in \R^+_0 \times Z$. Then $\ul{\psi}(s)$, $\ul{\omega}(s)$ are representatives of $\psi(s)$, $\partial_{\zeta} \psi(s)$ for every $s \in \R^+_0$ and 
\begin{align*}
(s,\zeta) \mapsto \ul{\psi}(s)(\zeta) \qquad \text{and} \qquad (s,\zeta) \mapsto \ul{\omega}(s)(\zeta)
\end{align*}
are continuous or measurable, respectively. So, by Tonelli's theorem and by %the compactness of $J = [r,t]$ and 
the continuity of~\eqref{eq:F db, def psi}, it follows that 
\begin{align} 
\int_{J\times Z} |\ul{\psi}(s)(\zeta)| \d (s,\zeta) &= \int_J \norm{\psi(s)}_{1} \d s \le \lambda(J) \sup_{s \in J}  \norm{\psi(s)}_{1} < \infty 
\label{eq:F db, step 3, ul-psi ib} \\
\int_{J\times Z} |\ul{\omega}(s)(\zeta)| \d (s,\zeta) &= \int_J \norm{\partial_{\zeta} \psi(s)}_{1} \d s \le \lambda(J) \sup_{s \in J}  \norm{\partial_{\zeta} \psi(s)}_{1} < \infty 
\label{eq:F db, step 3, ul-omega ib}
\end{align}
We can thus apply Fubini's theorem to see that for every $\phi \in C_c^{\infty}(Z^{\circ},\K)$
\begin{align} \label{eq:F db, step 3, berechnung schw abl von Phi_r,t}
\int_Z \phi'(\zeta) \Phi_{r,t}(\zeta) \d  \zeta &= \int_J \int_Z \phi'(\zeta) \ul{\psi}(s)(\zeta) \d \zeta \, \d s \notag \\
&= -\int_J \int_Z \phi(\zeta) \ul{\omega}(s)(\zeta) \d \zeta \, \d s = - \int_Z \phi(\zeta) \int_J \ul{\omega}(s)(\zeta) \d s \, \d \zeta.
\end{align}
So, by~\eqref{eq:F db, step 3, ul-omega ib} and~\eqref{eq:F db, step 3, berechnung schw abl von Phi_r,t} the map $\Phi_{r,t}$ is weakly differentiable with integrable weak derivative given by 
\begin{align} \label{eq:F db, step 3, schw abl von Phi_r,t}
\partial_{\zeta} \Phi_{r,t}(\zeta) = \int_J \ul{\omega}(s)(\zeta) \d s = \int_r^t \ul{\omega}(s)(\zeta) \d s
\end{align}
for almost every $\zeta \in Z$. Since $\int_r^t \dot{x}(s)^* P_1^{-1} x(s) + x(s)^* P_1^{-1} \dot{x}(s) \d s = x(s)^* P_1^{-1} x(s) |_{s=r}^{s=t}$ we have (Corollary~\ref{cor:mb representation of abstract functions}!)
\begin{align} \label{eq:F db, step 3, antiderivative}
\int_r^t \ul{v}(s)(\zeta)^* P_1^{-1} \ul{x}(s)(\zeta) + \ul{x}(s)(\zeta)^* P_1^{-1} \ul{v}(s)(\zeta) \d s
= \ul{x}(s)(\zeta)^* P_1^{-1} \ul{x}(s)(\zeta) \Big|_{s=r}^{s=t}
\end{align}
for almost every $\zeta \in Z$. Combining now~\eqref{eq:F db, step 3, schw abl von Phi_r,t} with~\eqref{eq:F db, def ul-omega} and~\eqref{eq:F db, step 3, antiderivative}, we obtain the desired formula~\eqref{eq:F db, step 3}. 
\smallskip

%Schritt 4
As a fourth step, we show that $F:Z \to \K$ is absolutely continuous. %and hence differentiable almost everywhere.
We immediately see from the second and third step, that for every $r, t \in \R^+_0$ the map $\Phi_{r,t}$ is absolutely continuous with
\begin{align} \label{eq:F db, step 4, 1}
\Phi_{r,t}(\zeta) = \Phi_{r,t}(\zeta_0) + \int_{\zeta_0}^{\zeta} \partial_{\eta} \Phi_{r,t}(\eta) \d \eta
= \Phi_{r,t}(\zeta_0) + \int_{\zeta_0}^{\zeta} \utilde{\Psi}_{r,t}(\eta) \d \eta + \int_{\zeta_0}^{\zeta} \widetilde{\Psi}_{r,t}(\eta) \d \eta
\end{align}
for every $\zeta, \zeta_0 \in Z$, where
\begin{align*}
\utilde{\Psi}_{r,t}(\eta) &:= \ul{x}(s)(\eta)^* P_1^{-1} \ul{x}(s)(\eta) \big|_{s=r}^{s=t} \notag \\
&\qquad - \int_r^t \ul{x}(s)(\eta)^* \big(  (P_1^{-1} P_0 \mathcal{H}(\eta))^{*} + P_1^{-1} P_0 \mathcal{H}(\eta)  \big) \ul{x}(s)(\eta) \d s 
%\label{eq:F db, step 4, def utilde Psi}
\\
\widetilde{\Psi}_{r,t}(\eta) &:= \int_r^t \ul{x}(s)(\eta)^*   \mathcal{H}'(\eta)  \ul{x}(s)(\eta) \d s.
%\label{eq:F db, step 4, def tilde Psi}
\end{align*}
We also have %by definition of~$\Phi_{r,t}$ and $\ul{\psi}$
\begin{align} \label{eq:F db, step 4, 2}
F(\zeta) 
&= \Phi_{r(\zeta),t(\zeta)}(\zeta) = \int_{r(\zeta)}^{t(\zeta)} \ul{\psi}(s)(\zeta) \d s \notag \\
&= \Phi_{r(\zeta_0),t(\zeta_0)}(\zeta) + \int_{t(\zeta_0)}^{t(\zeta)} \ul{\psi}(s)(\zeta) \d s -  \int_{r(\zeta_0)}^{r(\zeta)} \ul{\psi}(s)(\zeta) \d s
\end{align}
for every $\zeta, \zeta_0 \in Z$. So, by~\eqref{eq:F db, step 4, 1} and~\eqref{eq:F db, step 4, 2} we see that
\begin{align} \label{eq:F db, step 4, F(z)-F(z_0)}
F(\zeta) - F(\zeta_0) 
&= \int_{\zeta_0}^{\zeta} \utilde{\Psi}_{r(\zeta_0),t(\zeta_0)}(\eta) \d \eta + \int_{\zeta_0}^{\zeta} \widetilde{\Psi}_{r(\zeta_0),t(\zeta_0)}(\eta) \d \eta \notag \\
&\qquad + \int_{t(\zeta_0)}^{t(\zeta)} \ul{\psi}(s)(\zeta) \d s -  \int_{r(\zeta_0)}^{r(\zeta)} \ul{\psi}(s)(\zeta) \d s
\end{align}
for every $\zeta, \zeta_0 \in Z$. 
Choose now a compact interval $J$ such that
\begin{align} \label{eq:F db, step 4, r(z),t(z) in J}
r(\zeta), t(\zeta) \in J \qquad (\zeta \in Z)
\end{align}
($r, t$ are continuously differentiable on the compact interval $Z$ by assumption!).
With the help of~\eqref{eq:F db, step 4, r(z),t(z) in J} it then follows by the definition of $\utilde{\Psi}_{r,t}$, $\widetilde{\Psi}_{r,t}$, $\ul{\psi}$ %using~\eqref{eq:F db, step 4, r(z),t(z) in J} 
that
\begin{align}
\bigg| \int_{\zeta_0}^{\zeta} \utilde{\Psi}_{r(\zeta_0),t(\zeta_0)}(\eta) \d \eta \bigg|
&\le 
2 \big( \norm{P_1^{-1}} + \norm{P_1^{-1}P_0} \big) \norm{\ul{x}}_{J\times Z,\infty}^2 \cdot \notag \\
&\qquad \qquad \qquad \cdot \Big( 1+\sup_{\eta \in Z} \norm{\mathcal{H}(\eta)} \lambda(J) \Big) \, |\zeta-\zeta_0|
\label{eq:F db, step 4, 3.1}
\end{align}
\begin{align}
\bigg| \int_{\zeta_0}^{\zeta} \widetilde{\Psi}_{r(\zeta_0),t(\zeta_0)}(\eta) \d \eta \bigg|
&\le 
\norm{\ul{x}}_{J\times Z,\infty}^2 \lambda(J) \, \bigg| \int_{\zeta_0}^{\zeta} \norm{\mathcal{H}'(\eta)} \d \eta \bigg|
\label{eq:F db, step 4, 3.2}
\end{align}
\begin{align}
\bigg| \int_{t(\zeta_0)}^{t(\zeta)} \ul{\psi}(s)(\zeta) \d s \bigg|,   \bigg| \int_{r(\zeta_0)}^{r(\zeta)} \ul{\psi}(s)(\zeta) \d s \bigg|
&\le 
\norm{\ul{x}}_{J\times Z,\infty}^2 \sup_{\eta \in Z} \norm{\mathcal{H}(\eta)} \cdot \notag \\
& \qquad \qquad \cdot \max\{ \|r'\|_{\infty}, \|t'\|_{\infty} \} \, |\zeta-\zeta_0|
\label{eq:F db, step 4, 3.3}
\end{align} 
for all $\zeta, \zeta_0 \in Z$. Since $\mathcal{H}'$ is integrable and since $\norm{\ul{x}}_{J\times Z,\infty} := \sup_{(s,\zeta)\in J\times Z} |\ul{x}(s)(\zeta)| < \infty$ and $\|r'\|_{\infty}, \|t'\|_{\infty} < \infty$ by the continuity of~\eqref{eq:F db, step 2, 1} and by  assumption respectively, %and $r,t$ are continuously differentiable and $\ul{x}$ is %continuous and hence
%bounded on $J \times Z$ by~\eqref{eq:F db, step 2, 1}, 
it follows from~\eqref{eq:F db, step 4, F(z)-F(z_0)} with the help of~\eqref{eq:F db, step 4, 3.1}, \eqref{eq:F db, step 4, 3.2}, \eqref{eq:F db, step 4, 3.3} that $F$ is absolutely continuous, as desired. 
\smallskip

%Schritt 5
As a fifth step, we show that the derivative of $F$ -- which by the fourth step exists almost everywhere -- is given by the asserted formula~\eqref{eq:abl von F} for almost every $\zeta$. 
Since $\utilde{\Psi}_{r,t}$ and $(s,\zeta) \mapsto \ul{\psi}(s)(\zeta)$ are continuous, it follows that
\begin{align}  \label{eq:F db, step 5, 1}
\frac{1}{\zeta-\zeta_0} \int_{\zeta_0}^{\zeta} \utilde{\Psi}_{r(\zeta_0),t(\zeta_0)}(\eta) \d \eta \longrightarrow \utilde{\Psi}_{r(\zeta_0),t(\zeta_0)}(\zeta_0) \qquad (\zeta \to \zeta_0)
\end{align}
for every $\zeta_0 \in Z$ and that
\begin{align} \label{eq:F db, step 5, 2}
&\frac{1}{\zeta-\zeta_0} \int_{t(\zeta_0)}^{t(\zeta)} \ul{\psi}(s)(\zeta) \d s - \frac{1}{\zeta-\zeta_0} \int_{r(\zeta_0)}^{r(\zeta)} \ul{\psi}(s)(\zeta) \d s \notag \\
&\qquad \qquad \longrightarrow 
t'(\zeta_0) \, \ul{\psi}(s)(\zeta_0) \big|_{s=t(\zeta_0)} - r'(\zeta_0) \, \ul{\psi}(s)(\zeta_0) \big|_{s=r(\zeta_0)} \qquad (\zeta \to \zeta_0)
\end{align}
for every $\zeta_0 \in Z$.
%We now show that an analog of~\eqref{eq:F db, step 5, 1} for $\widetilde{\Psi}_{r(\zeta_0),t(\zeta_0)}$ holds (at least) for almost every $\zeta_0 \in Z$. Indeed, choose ...
Choose now a null set $N$ such that
\begin{align*} %\label{eq:F db, step 5, def null set}
\frac{1}{\zeta-\zeta_0} \int_{\zeta_0}^{\zeta}  \norm{ \mathcal{H}'(\eta) - \mathcal{H}'(\zeta_0)} \d \eta \longrightarrow 0 \qquad (\zeta \to \zeta_0)
\end{align*} 
for all $\zeta_0 \in Z \setminus N$, which is possible by the integrability of $\mathcal{H}'$ and Lebesgue's differentiation theorem. Since by the definition of $\widetilde{\Psi}_{r,t}$
\begin{align*}
&\bigg| \frac{1}{\zeta-\zeta_0} \int_{\zeta_0}^{\zeta} \widetilde{\Psi}_{r(\zeta_0),t(\zeta_0)}(\eta) - \widetilde{\Psi}_{r(\zeta_0),t(\zeta_0)}(\zeta_0) \d \eta \bigg| \\
&\qquad \qquad \le 
\norm{\ul{x}}_{J\times Z,\infty}^2 \lambda(J) \, \bigg| \frac{1}{\zeta-\zeta_0} \int_{\zeta_0}^{\zeta}  \norm{ \mathcal{H}'(\eta) - \mathcal{H}'(\zeta_0)} \d \eta \bigg| \\
&\quad \qquad \qquad + 2 \norm{\ul{x}}_{J\times Z,\infty} \norm{\mathcal{H}'(\zeta_0)} \bigg| \frac{1}{\zeta-\zeta_0} \int_{\zeta_0}^{\zeta} \int_{r(\zeta_0)}^{t(\zeta_0)} | \ul{x}(s)(\eta) - \ul{x}(s)(\zeta_0) | \d s \d \eta  \bigg|
\end{align*}
for every $\zeta, \zeta_0 \in Z$ (with $J$ as in~\eqref{eq:F db, step 4, r(z),t(z) in J}), it follows that
\begin{align}  \label{eq:F db, step 5, 3}
\frac{1}{\zeta-\zeta_0} \int_{\zeta_0}^{\zeta} \widetilde{\Psi}_{r(\zeta_0),t(\zeta_0)}(\eta) \d \eta \longrightarrow \widetilde{\Psi}_{r(\zeta_0),t(\zeta_0)}(\zeta_0) \qquad (\zeta \to \zeta_0)
\end{align}
for every $\zeta_0 \in Z \setminus N$. 
Combining now~\eqref{eq:F db, step 5, 1}, \eqref{eq:F db, step 5, 2}, \eqref{eq:F db, step 5, 3} with~\eqref{eq:F db, step 4, F(z)-F(z_0)}, we conclude that $F$ is differentiable at every $\zeta_0 \in Z \setminus N$ with derivative
\begin{align} \label{eq:F db, step 5, 4}
F'(\zeta_0) 
&= \utilde{\Psi}_{r(\zeta_0),t(\zeta_0)}(\zeta_0) + \widetilde{\Psi}_{r(\zeta_0),t(\zeta_0)}(\zeta_0) + t'(\zeta_0) \, \ul{\psi}(s)(\zeta_0) \big|_{s=t(\zeta_0)} - r'(\zeta_0) \, \ul{\psi}(s)(\zeta_0) \big|_{s=r(\zeta_0)} \notag \\
&= \partial_{\zeta} \Phi_{r(\zeta_0),t(\zeta_0)}(\zeta) \big|_{\zeta = \zeta_0} + t'(\zeta_0) \, \ul{\psi}(s)(\zeta_0) \big|_{s=t(\zeta_0)} - r'(\zeta_0) \, \ul{\psi}(s)(\zeta_0) \big|_{s=r(\zeta_0)}
\end{align}
for every $\zeta_0 \in Z \setminus N$. 
In view of
%\begin{align*}
%\utilde{\Psi}_{r(\zeta_0),t(\zeta_0)}(\zeta_0) + \widetilde{\Psi}_{r(\zeta_0),t(\zeta_0)}(\zeta_0)
%= \partial_{\zeta} \Phi_{r(\zeta_0),t(\zeta_0)}(\zeta) \big|_{\zeta = \zeta_0}
%\end{align*}
%and of the formula~
\eqref{eq:F db, step 3} from the third step, this %\eqref{eq:F db, step 5, 4} 
is precisely the asserted formula~\eqref{eq:abl von F}  for the derivative (with $\zeta$ replaced by $\zeta_0$). 
\smallskip

(ii) We finally show -- by some slight modifications of the arguments above -- that $F$ is even continuously differentiable under the strengthened assumption that $$\mathcal{H} \in C^1(Z,\K^{m\times m}).$$ So, let $\mathcal{H} \in C^1(Z,\K^{m\times m})$. We can then argue until~\eqref{eq:F db, step 4, F(z)-F(z_0)} in exactly the same way as above. And this equation~\eqref{eq:F db, step 4, F(z)-F(z_0)}, under our strengthened assumption, almost immediately yields the desired conclusion. Indeed, for $\mathcal{H} \in C^1(Z,\K^{m\times m})$ not only $\utilde{\Psi}_{r,t}$, $(s,\zeta) \mapsto \ul{\psi}(s)(\zeta)$ but also $\widetilde{\Psi}_{r,t}$ is continuous and therefore not only~\eqref{eq:F db, step 5, 1}, \eqref{eq:F db, step 5, 2} but also~\eqref{eq:F db, step 5, 3} holds true for every $\zeta_0 \in Z$. So, from~\eqref{eq:F db, step 4, F(z)-F(z_0)} we see that $F$ is differentiable at every $\zeta_0 \in Z$ with derivative given by~\eqref{eq:F db, step 5, 4}. %which is continuous in $\zeta_0$ under our strengthened assumption. 
And this expression, in turn, is continuous in $\zeta_0$ under our strengthened assumption. 
\end{proof}

%If $\zeta \mapsto \mathcal{H}(\zeta)$ is assumed to be even continuously differentiable (as is done in~\cite{ViZwLeMa09} and~\cite{JacobZwart}), then the map $F$ is continuously differentiable everywhere: the set $N$ from the above proof is empty. Indeed,  in that special situation from~\cite{ViZwLeMa09} and~\cite{JacobZwart}, both $\utilde{\Psi}_{r,t}$ and $\tilde{\Psi}_{r,t}$ are continuous and therefore~\eqref{eq:F db, step 4, 1} immediately implies that the map $\Phi_{r,t}$ and thus also $F$ is continuously differentiable at every $\zeta \in Z$. %So, in the special situation of~\cite{ViZwLeMa09} and~\cite{JacobZwart}, the proof of the fourth step from the above proof becomes much simpler and $N = \emptyset$. 
%\smallskip
%
%We now show by example that the correctness of the formulas in the above proof and in~\cite{ViZwLeMa09}, \cite{JacobZwart} essentially hinges on a careful choice of representatives.

\begin{ex}
Choose $A$ to be the port-Hamiltonian operator on $X := L^2(Z,\R)$ corresponding to the transport equation on $Z := [0,1]$, that is,  
\begin{align}
A = \partial_{\zeta} \qquad \text{with} \qquad D(A) = W^{1,2}(Z^{\circ},\R)
\end{align}
and thus $\mathcal{H}(\zeta) = 1 \in \R$ for all $\zeta \in Z$ and $P_1 = 1, P_0 = 0 \in \R$.  
In particular, $s \mapsto x(s) := 1$ is a classical solution of $\dot{x} = Ax$. 
Choose now $\ul{x}(s), \ul{v}(s): Z \to \R$ for $s \in J := [0,1]$ in the following way:
\begin{align}
\ul{x}(s)(\zeta) := 1 \qquad \text{and} \qquad \ul{v}(s)(\zeta) := \chi_{E}(s,\zeta)
\end{align}
for every $(s,\zeta) \in J \times Z$, where $E$ is chosen as in Example~\ref{ex:mb representation of abstract functions}~(ii).
%as a subset of $J \times Z$ such that the section $E_s := \{ \zeta \in Z: (s,\zeta) \in E\}$ is countable for every $s \in J$ and that the section $E^{\zeta} := \{s \in J: (s,\zeta) \in E\}$ is co-countable for every $\zeta \in Z$. (Such a set $E$ was shown to exist by Sierpinski assuming that the continuum hypothesis is true. See~\cite{Si20b} or Example 8.9~(c) in~\cite{Rudin:real-complex}.)
We then have that $\ul{x}(s)$ for every $s \in J$ is the continuous representative of $x(s)$ and that $\ul{v}(s)$ for every $s \in J$ is a representative of $\dot{x}(s)$, 
but with this specific choice of representatives %the calculation of the first integral from the last equation on page~113 of~\cite{JacobZwart} becomes false:
the formula~\eqref{eq:F db, step 3, antiderivative} -- and hence the formula for the first integral from the last equation on page~113 of~\cite{JaZw} -- becomes false. Indeed,
\begin{align}
\int_0^1 \ul{v}(s)(\zeta)^* P_1^{-1} \ul{x}(s)(\zeta) + \ul{x}(s)(\zeta)^* P_1^{-1} \ul{v}(s)(\zeta) \d s
&= 2 \int_0^1 \chi_E(s,\zeta) \d s = 2 \notag \\
&\ne 0 = \ul{x}(s)(\zeta)^* P_1^{-1} \ul{x}(s)(\zeta) \Big|_{s=0}^{s=1}
\end{align}
for every $\zeta \in Z$. $\blacktriangleleft$
\end{ex}

%VORBEREITENDES LEMMA 3 (APPROXIMATION VON \mathcal{H} \in BV DURCH \mathcal{H}_n \in AC)

As a third %and final 
preparatory lemma, we show %will need 
the following approximation result for an energy density $\mathcal{H}$ of bounded variation by absolutely continuous energy densities $\mathcal{H}_n$.

\begin{lm} \label{lm:approx-en-density}
Suppose $\mathcal{H} \in BV([a,b],\K^{m\times m})$ is an energy density with lower and upper bounds denoted by $\ul{m}, \ol{m}$. Then there exists a sequence of energy densities $\mathcal{H}_n \in AC([a,b],\K^{m\times m})$ %and a constant $\ol{m}' \in (0,\infty)$
such that
\begin{itemize}
\item[(i)] $\mathcal{H}_n(\zeta) \longrightarrow \mathcal{H}(\zeta)$ as $n \to \infty$ for almost every $\zeta \in [a,b]$
\item[(ii)] $\ul{m} \le \mathcal{H}_n(\zeta) \le \ol{m}$ for every $\zeta \in [a,b]$ and every $n \in \N$
\item[(iii)] $\int_a^b \norm{\mathcal{H}_n'(\zeta)} \d \zeta \le \norm{\mathcal{H}(a)} + \operatorname{Var}(\mathcal{H}) + \norm{\mathcal{H}(b)}$ for all $n \in \N$. %$\int_a^b \norm{\mathcal{H}_n'(\zeta)} \d \zeta \le \ol{m}'$ for all $n \in \N$.
\end{itemize}
\end{lm}

\begin{proof}
We argue by mollification. So, let $j \in C_c^{\infty}(\R)$ be such that
\begin{align} \label{eq:approxlm, 1}
j(r) \ge 0 \qquad (r \in \R) \qquad \text{and} \qquad \int_{\R} j(r) \d r = 1
\end{align}
and let $j_{\eps}(r) := 1/\eps \cdot  j(r/\eps)$ for $r \in \R$ and $\eps >0$. 
Since $\mathcal{H} \in BV([a,b],\K^{m\times m}) \subset L^1([a,b],\K^{m\times m})$, it follows that 
$j_{\eps} * \mathcal{H} \in C_c^{\infty}(\R,\K^{m\times m})$ and that $j_{\eps} * \mathcal{H} \longrightarrow \mathcal{H}$ in $L^1$ as $\eps \searrow 0$. In particular, there exists a sequence $(\eps_n)$ such that $\eps_n \searrow 0$ and
\begin{align} \label{eq:approxlm, 2}
(j_{\eps_n} * \mathcal{H})(\zeta) \longrightarrow \mathcal{H}(\zeta) \qquad (n \to \infty)
\end{align} 
for almost every $\zeta \in [a,b]$. Setting now 
\begin{align*}
\mathcal{H}_n := (j_{\eps_n} * \mathcal{H})|_{[a,b]},
\end{align*}
we obtain %first observe
$\mathcal{H}_n \in C^{\infty}([a,b],\K^{m\times m}) \subset AC([a,b],\K^{m\times m})$ for all $n \in \N$.
Also, assertion~(i) follows from~\eqref{eq:approxlm, 2} and assertion~(ii) follows from~\eqref{eq:en-density, lower and upper bound} using~\eqref{eq:approxlm, 1}.
It remains to prove assertion~(iii). Since $\mathcal{H}_n \in C^1([a,b],\K^{m\times m})$, it follows by a well-known formula for curve lengths (Theorem~VIII.1.3 of~\cite{AmEs}) that
\begin{align} \label{eq:approxlm, 3}
\int_a^b \norm{\mathcal{H}_n'(\zeta)} \d \zeta = \operatorname{Var}(\mathcal{H}_n)
\end{align}
for every $n \in \N$. Since, moreover, $\mathcal{H} \in BV([a,b],\K^{m\times m})$, it follows that %the zero-extension $\tilde{\mathcal{H}}$ of $\mathcal{H}$ satisfies
\begin{align*}
\tilde{\mathcal{H}} \in BV(\R,\K^{m\times m}) 
\qquad \text{and} \qquad
\operatorname{Var}(\tilde{\mathcal{H}}) = \norm{\mathcal{H}(a)} + \operatorname{Var}(\mathcal{H}) + \norm{\mathcal{H}(b)},
\end{align*}
%$\tilde{\mathcal{H}} \in BV(\R,\K^{m\times m})$ and $\operatorname{Var}(\tilde{\mathcal{H}}) = \norm{\mathcal{H}(a)} + \operatorname{Var}(\mathcal{H}) + \norm{\mathcal{H}(b)}$
where $\tilde{\mathcal{H}}: \R \to \K^{m\times m}$ is the zero-extension of $\mathcal{H}$. So, for every partition $(t_l)_{l \in \{0,\dots, L\}}$ of $[a,b]$, we see using~\eqref{eq:approxlm, 1} that
\begin{align*}
&\sum_{l=1}^L \norm{ \mathcal{H}_n(t_l) - \mathcal{H}_n(t_{l-1}) }
= \sum_{l=1}^L \norm{ \int_{\R} j_{\eps_n}(r) \big( \tilde{\mathcal{H}}(t_l-r) - \tilde{\mathcal{H}}(t_{l-1}-r) \big) \d r } \notag \\
&\qquad \le \int_{\R} j_{\eps_n}(r) \sum_{l=1}^L \norm{ \tilde{\mathcal{H}}(t_l-r) - \tilde{\mathcal{H}}(t_{l-1}-r) } \d r 
%\notag \\
\le \operatorname{Var}(\tilde{\mathcal{H}}) \notag \\
&\qquad = \norm{\mathcal{H}(a)} + \operatorname{Var}(\mathcal{H}) + \norm{\mathcal{H}(b)}
\end{align*}
and therefore
\begin{align} \label{eq:approxlm, 4}
\operatorname{Var}(\mathcal{H}_n) \le \norm{\mathcal{H}(a)} + \operatorname{Var}(\mathcal{H}) + \norm{\mathcal{H}(b)}
\end{align}
for every $n \in \N$. Combining now~\eqref{eq:approxlm, 3} and~\eqref{eq:approxlm, 4}, we obtain the desired conclusion~(iii) and we are done. 
\end{proof}

%STABILITAETSSATZ

With the above lemmas at hand, we can now show the following exponential stability result for port-Hamiltonian operators with energy densities of bounded variation. It is a generalization of the respective stability results from~\cite{JaZw} (Theorem~9.1.3) and~\cite{Au} (Theorem~4.1.5) where the energy densities are required to be continuously differentiable or Lipschitz continuous, respectively.

\begin{thm} \label{thm:stab-thm}
Suppose $A: D(A) \subset X \to X$ is a first-order port-Hamiltonian operator with energy density $\mathcal{H} \in BV([a,b],\K^{m\times m})$, where $X := L^2([a,b],\K^m)$ is endowed with the $\mathcal{H}$-energy norm $\norm{\cdot}_X$. Suppose further that the domain of $A$ is of the form
\begin{align} \label{eq:stab-thm, domain-ass}
D(A) = \big\{ x \in X: \mathcal{H}x \in W^{1,2}((a,b),\K^m) \text{ and } W(\mathcal{H}x)|_{\partial} = 0 \big\}
\end{align}
for some matrix $W \in \K^{m\times 2m}$ and that there exists $\kappa \in (0,\infty)$ such that for $c = a$ or $c = b$ one has
\begin{align} \label{eq:stab-thm, dissip-ass}
\Re \scprd{x,Ax}_X \le -\kappa |(\mathcal{H}x)(c)|^2 \qquad (x \in D(A)).
\end{align}
Then $A$ generates an exponentially stable contraction semigroup on $X$. 
\end{thm}

\begin{proof}
It immediately follows from the assumption~\eqref{eq:stab-thm, dissip-ass} by Lemma~\ref{lm:char-contr-sgr-gen-property} that $A$ generates a contraction semigroup on $X$ and so we have only to %it suffices to 
show that $\e^{A \cdot}$ is exponentially stable. We do so in various steps by means of a suitable approximation argument. We write 
\begin{align*}
Jf := P_1 \partial_{\zeta}f + P_0 f \qquad (f \in D(J)),
\end{align*}
where $P_1, P_0$ are the matrices defining $A$ and where
\begin{align*}
D(J) := \{ f \in W^{1,2}((a,b),\K^m): W f|_{\partial} = 0 \}.
\end{align*} 
In particular, we have $A = J \mathcal{H}$. We also choose energy densities $\mathcal{H}_n \in AC([a,b],\K^{m\times m})$ as in Lemma~\ref{lm:approx-en-density}, define
\begin{align*}
A_n := J \mathcal{H}_n,
\end{align*}
and endow $X_n := L^2([a,b],\K^m)$ with the $\mathcal{H}_n$-energy norm $\norm{\cdot}_{X_n}$. In particular, %we have
\begin{align*}
D(A_n) = \big\{ x \in X_n: \mathcal{H}_n x \in W^{1,2}((a,b),\K^m) \text{ and } W (\mathcal{H}_n x)|_{\partial} = 0 \big\}.
\end{align*}

%Schritt 1
As a first step, we show that $A_n$ is a contraction semigroup generator on $X_n$ for every $n \in \N$. 
Indeed, $A_n$ is a port-Hamiltonian operator with energy density $\mathcal{H}_n$ 
which has a domain of the form~\eqref{eq:domain-with-m-lin-bdry-cond} and is dissipative in $X_n$. In order to see the dissipativity, note that for every $x_n \in D(A_n) = D(J\mathcal{H}_n)$ one has $f_n := \mathcal{H}_n x_n \in D(J)$ and therefore
\begin{align*}
\mathcal{H}_n x_n = f_n = \mathcal{H}y_n
\end{align*}
for $y_n := \mathcal{H}^{-1} f_n \in D(J\mathcal{H}) = D(A)$. So, by the assumption~\eqref{eq:stab-thm, dissip-ass}, we have for $c = a$ or $c=b$ that %and for every $x_n \in D(A_n)$ that 
\begin{align} \label{eq:stab-thm, dissip-cond for A_n}
\Re \scprd{x_n,A_n x_n}_{X_n} &= \Re \scprd{\mathcal{H}_n x_n, J \mathcal{H}_n x_n}_2 = \Re \scprd{\mathcal{H} y_n, J \mathcal{H} y_n}_2
= \Re \scprd{y_n, Ay_n}_X \notag \\
&\le -\kappa |(\mathcal{H}y_n)(c)|^2 = -\kappa |(\mathcal{H}_n x_n)(c)|^2 %\le 0
\end{align}
for every $x_n \in D(A_n)$, which implies the claimed dissipativity of $A_n$ in $X_n$. In view of Lemma~\ref{lm:char-contr-sgr-gen-property} this %proves the first step.
concludes the proof of the first step.
\smallskip

%Schritt 2
As a second step, we show that there exist constants $\gamma_0, \kappa_0 \in (0,\infty)$ such that for every $n \in \N$ and $x_{n 0} \in D(A_n)$ one has the following sideways energy estimates:
\begin{align} \label{eq:stab-thm, step-2, sideways-energy-estimate}
F_{n \tau}^+(\zeta) \le F_{n \tau}^+(b)\,  \e^{\kappa_0 (b-a)}
\qquad \text{and} \qquad 
F_{n \tau}^-(\zeta) \le F_{n \tau}^-(a) \, \e^{\kappa_0 (b-a)}
\end{align}
for every $\zeta \in [a,b]$ and every $\tau > 2 \gamma_0 (b-a)$, where
\begin{align*}
&F_{n \tau}^+(\zeta) := \int_{\gamma_0(b-\zeta)}^{\tau - \gamma_0(b-\zeta)} \ul{x}_n(s)(\zeta)^* \mathcal{H}_n(\zeta) \ul{x}_n(s)(\zeta) \d s \\
&F_{n \tau}^-(\zeta) := \int_{\gamma_0(\zeta-a)}^{\tau - \gamma_0(\zeta-a)} \ul{x}_n(s)(\zeta)^* \mathcal{H}_n(\zeta) \ul{x}_n(s)(\zeta) \d s
\end{align*}
and where $\ul{x}_n(s)$ denotes the continuous representative of $x_n(s) := \e^{A_n \cdot} x_{n 0}$. 
We can argue similarly to~\cite{JaZw}, \cite{Au}, the essential difference being that in contrast to~\cite{JaZw}, \cite{Au} the derivative $\mathcal{H}_n'$ here need not be in $L^{\infty}$ but is  only in $L^1$. 
Set 
\begin{align} \label{eq:stab-thm, step-2, def gamma und kappa}
\gamma_0 := \norm{P_1^{-1}} / \ul{m}
\qquad \text{and} \qquad
\kappa_0 := \big( 2 \norm{P_1^{-1} P_0} \ol{m} + \ol{m}' \big) / \ul{m}
\end{align}
where $\ul{m}, \ol{m}$ are as in Lemma~\ref{lm:approx-en-density} and $\ol{m}' := \norm{\mathcal{H}(a)} + \operatorname{Var}(\mathcal{H}) + \norm{\mathcal{H}(b)}$. Also, choose and fix $n \in \N$ and $x_{n 0} \in D(A_n)$ and write $x_n := \e^{A_n \cdot} x_{n 0}$. Since $A_n$ is a port-Hamiltonian operator with energy density $\mathcal{H}_n \in AC([a,b],\K^{m\times m})$ and since $x_n = \e^{A_n \cdot} x_{n 0}$ is a classical solution of 
\begin{align*}
\dot{x} = A_n x,
\end{align*}
it follows by Lemma~\ref{lm:F db a.e.} that $F_{n \tau}^{\pm}$ for every $\tau > 2 \gamma_0 (b-a)$ is absolutely continuous and hence differentiable almost everywhere with derivative given by
\begin{align} \label{eq:stab-thm, step 2, abl F_tau}
%\partial_{\zeta} F_{\tau}^{\pm}(\zeta) 
(F_{n \tau}^{\pm})'(\zeta)
&= \ul{x}_n(s)(\zeta)^* \big( \pm \gamma_0 \mathcal{H}_n(\zeta) + P_1^{-1} \big) \ul{x}_n(s)(\zeta) \Big|_{s=t^{\pm}(\zeta)} \notag \\
&\quad + \ul{x}_n(s)(\zeta)^* \big( \pm \gamma_0 \mathcal{H}_n(\zeta) - P_1^{-1} \big) \ul{x}_n(s)(\zeta) \Big|_{s=r^{\pm}(\zeta)}  \\ 
&\quad - \int_{r^{\pm}(\zeta)}^{t^{\pm}(\zeta)} \ul{x}_n(s)(\zeta)^* \Big( (P_1^{-1} P_0 \mathcal{H}_n(\zeta))^* + \mathcal{H}_n'(\zeta) + P_1^{-1} P_0 \mathcal{H}_n(\zeta) \Big) \ul{x}_n(s)(\zeta) \d s \notag
\end{align} 
for a.e.~$\zeta \in [a,b]$, where $r^+(\zeta) := \gamma_0(b-\zeta)$, $t^+(\zeta) := \tau - \gamma_0(b-\zeta)$ and $r^-(\zeta) := \gamma_0(\zeta-a)$, $t^-(\zeta) := \tau - \gamma_0(\zeta-a)$.
%\begin{align*}
%&r^+(\zeta) := \gamma(b-\zeta), \qquad t^+(\zeta) := \tau - \gamma(b-\zeta) \\
%%\qquad \text{and} \qquad
%&r^-(\zeta) := \gamma(\zeta-a), \qquad t^-(\zeta) := \tau - \gamma(\zeta-a).
%\end{align*}
In view of Lemma~\ref{lm:approx-en-density}~(ii) and of~(\ref{eq:stab-thm, step-2, def gamma und kappa}.a) it follows from~\eqref{eq:stab-thm, step 2, abl F_tau} that
\begin{align}
(F_{n \tau}^+)'(\zeta) &\ge - \kappa_n(\zeta) \int_{r^+(\zeta)}^{t^+(\zeta)} \ul{m} \, |\ul{x}_n(s)(\zeta)|^2 \d s 
\ge - \kappa_n(\zeta) F_{n \tau}^+(\zeta) 
\label{eq:stab-thm, step 2, diffungl fuer F_tau^+} \\
(F_{n \tau}^-)'(\zeta) &\le \kappa_n(\zeta) \int_{r^-(\zeta)}^{t^-(\zeta)} \ul{m} \, |\ul{x}_n(s)(\zeta)|^2 \d s 
\le \kappa_n(\zeta) F_{n \tau}^-(\zeta) 
\label{eq:stab-thm, step 2, diffungl fuer F_tau^-}
\end{align}
for all $\tau > 2 \gamma_0(b-a)$ and a.a.~$\zeta \in [a,b]$, where
\begin{align*}
\kappa_n(\zeta) := \big( 2\norm{P_1^{-1} P_0} \ol{m} + \norm{\mathcal{H}_n'(\zeta)} \big)/\ul{m}.
\end{align*}
Since $F_{n \tau}^{\pm}$ is absolutely continuous, the differential inequalities~\eqref{eq:stab-thm, step 2, diffungl fuer F_tau^+},  and~\eqref{eq:stab-thm, step 2, diffungl fuer F_tau^-} imply that $\zeta \mapsto F_{n \tau}^+(\zeta) \exp(-\int_{\zeta}^b \kappa_n(\eta) \d \eta)$ and $\zeta \mapsto F_{n \tau}^-(\zeta) \exp(-\int_a^{\zeta}\kappa_n(\eta) \d \eta)$ are monotonically increasing or decreasing, respectively. Consequently, 
\begin{align*}
&(F_{n \tau}^+)(\zeta) \le (F_{n \tau}^+)(b) \, \e^{\int_a^b \kappa_n(\eta) \d \eta} \le (F_{n \tau}^+)(b) \, \e^{\kappa_0 (b-a)} \\
&(F_{n \tau}^-)(\zeta) \le (F_{n \tau}^-)(a) \, \e^{\int_a^b \kappa_n(\eta) \d \eta} \le (F_{n \tau}^-)(a) \, \e^{\kappa_0 (b-a)}
\end{align*}
as desired, where for the second inequalities
Lemma~\ref{lm:approx-en-density}~(iii) has been used. 
\smallskip

As a third step, we show that there exist constants $C_0, t_0 \in (0,\infty)$ such that for every $n \in \N$ and $x_{n 0} \in D(A_n)$ one has the following estimate:
\begin{align} \label{eq:stab-thm, step 3}
\norm{x_n(\tau)}_{X_n}^2 \le C_0 \int_0^{\tau} |(\mathcal{H}_n x_n(s))(c)|^2 \d s
\end{align}
for every $\tau \ge t_0$ and for $c =a$ and $c=b$, where $x_n := \e^{A_n \cdot} x_{n 0}$. 
We can argue as in~\cite{JaZw}, \cite{Au} building on our second and  third step. %using the second and the third step. %in view of the
Set
\begin{align}
t_0 := 2 \gamma_0 (b-a) + 1 \qquad \text{and} \qquad C_0 := \frac{\e^{\kappa_0 (b-a)}}{2 \ul{m}} \, (b-a).
\end{align} 
Also, choose and fix $n \in \N$ and $x_{n 0} \in D(A_n)$ and write $x_n := \e^{A_n \cdot} x_{n 0}$. Since $A_n$ generates a contraction semigroup on $X_n$ by the first step, we see for every $\tau \ge t_0$ that
\begin{align} \label{eq:stab-thm, step 3, absch 1}
\norm{x_n(\tau)}_{X_n}^2 
&\le (\tau-2\gamma_0(b-a)) \norm{x_n(\tau)}_{X_n}^2 = \int_{\gamma_0(b-a)}^{\tau-\gamma_0(b-a)} \norm{x_n(\tau)}_{X_n}^2 \d s \notag \\
&\le \int_{\gamma_0(b-a)}^{\tau-\gamma_0(b-a)} \norm{x_n(s)}_{X_n}^2 \d s \notag \\
&= \frac{1}{2} \int_a^b \int_{\gamma_0(b-a)}^{\tau-\gamma_0(b-a)} \ul{x}_n(s)(\zeta)^* \mathcal{H}_n(\zeta)   \ul{x}_n(s)(\zeta) \d s \d \zeta,
\end{align} 
where interchanging the integrals in the last equality is justified due to the continuity of $(s,\zeta) \mapsto  \ul{x}_n(s)(\zeta)$, see~\eqref{eq:F db, step 2, 1}. Increasing the inner integration interval in~\eqref{eq:stab-thm, step 3, absch 1} to $[  \gamma_0(b-\zeta),  \tau - \gamma_0(b-\zeta) ]$ or $[ \gamma_0(\zeta-a),  \tau - \gamma_0(\zeta-a) ]$ respectively and using the sideways energy estimates~\eqref{eq:stab-thm, step-2, sideways-energy-estimate} from the second step, we conclude that
\begin{align}
\norm{x_n(\tau)}_{X_n}^2  \le \frac{1}{2} \int_a^b F_{n \tau}^{\pm}(\zeta) \d \zeta
\le \frac{1}{2} \min \{  F_{n \tau}^{+}(b),  F_{n \tau}^{-}(a) \} \, \e^{\kappa_0 (b-a)} (b-a)  
\end{align}
for every $\tau \ge t_0$.  And from this, in turn, the desired estimate~\eqref{eq:stab-thm, step 3} immediately follows (using the definition of~$F_n^{\pm}$) both for $c = b$ and for $c = a$. %using Lemma~\ref{lm:approx-en-density}~(ii). 
\smallskip

As a fourth step, we show that there exist constants $M_0 \in [1,\infty)$ and $\omega_0 \in (-\infty,0)$ such that
\begin{align} \label{eq:stab-thm, step 4}
\norm{\e^{A_n t}}_{X_n, X_n} \le M_0 \, \e^{\omega_0 t}  %\qquad (t \in \R^+_0 \text{ and } n \in \N)
\end{align}
for all $t \in \R^+_0$ and $n \in \N$, where $\norm{\cdot}_{X_n,X_n}$ is the operator norm induced by $\norm{\cdot}_{X_n}$. 
Indeed, from the third step and~\eqref{eq:stab-thm, dissip-cond for A_n} it follows that for every $n \in \N$ and $x_{n 0} \in D(A_n)$ 
\begin{align*}
\norm{ \e^{A_n t_0}x_{n 0}}_{X_n}^2 %= \norm{x_n(t_0)}_{X_n}^2 
&\le C_0 \int_0^{t_0} |(\mathcal{H}_n x_n(s))(c)|^2 
\le - C_0/\kappa \int_0^{t_0} \Re \scprd{x_n(s), A_n x_n(s)}_{X_n} \d s \\
&= C_0/(2\kappa) \big( \norm{x_{n 0}}_{X_n}^2 - \norm{ \e^{A_n t_0}x_{n 0}}_{X_n}^2 \big),
\end{align*}
where as usual $x_n := \e^{A_n \cdot} x_{n 0}$. 
So, by the density of $D(A_n)$ in $X_n$, we obtain %conclude
\begin{align*}
\norm{ \e^{A_n t_0} }_{X_n,X_n} \le \bigg( \frac{C_0/(2\kappa)}{1+C_0/(2\kappa)} \bigg)^{1/2} =: \mu_0
\end{align*}
for every $n \in \N$. And from this, in turn, we conclude by the semigroup and the contraction semigroup property of $\e^{A_n \cdot}$ that for arbitrary $t \in \R^+_0$ one has
\begin{align*}
\norm{ \e^{A_n t} }_{X_n,X_n} = \norm{ (\e^{A_n  t_0})^l \, \e^{A_n (t-lt_0) } }_{X_n,X_n} \le \mu_0^l = \frac{1}{\mu_0} \mu_0^{l+1} \le \frac{1}{\mu_0} \mu_0^{t/t_0},
\end{align*}
where we used the abbreviation $l := \lfloor t/t_0 \rfloor$ for the integer part of $t/t_0$ and the fact that $\mu_0 < 1$. Setting
\begin{align}
M_0 := \frac{1}{\mu_0} \in [1,\infty) \qquad \text{and} \qquad \omega_0 := (\log \mu_0)/t_0 \in (-\infty,0),
\end{align} 
we finally obtain the desired estimate~\eqref{eq:stab-thm, step 4}. 
\smallskip

As a fifth and last step, we can finally show that $\e^{A \cdot}$ is exponentially stable.
Indeed, since $\norm{\cdot}_{X_n}$ is equivalent to $\norm{\cdot}_X$ with equivalence constants independent of $n$ (Lemma~\ref{lm:approx-en-density}~(ii)!), it follows from the fourth step that there exists a constant $M \in [1,\infty)$ such that
\begin{align} \label{eq:stab-thm, step 5, 1}
\norm{\e^{A_n t}}_{X, X} \le M \, \e^{\omega_0 t}  %\qquad (t \in \R^+_0 \text{ and } n \in \N)
\end{align}
for all $t \in \R^+_0$ and $n \in \N$, where $\norm{\cdot}_{X,X}$ is the operator norm induced by $\norm{\cdot}_{X}$. %, where $\omega_0$ is the (negative) growth exponent from~\eqref{eq:stab-thm, step 4}
Also, for every $x \in D(A)$ there exists a sequence $(x_n)$ with $x_n \in D(A_n)$ and
\begin{align} \label{eq:stab-thm, step 5, 2}
x_n \underset{X}{\longrightarrow} x \qquad \text{and} \qquad A_n x_n \underset{X}{\longrightarrow} A x
\end{align}
as $n \to \infty$. (Indeed, for $x \in D(A) = D(J\mathcal{H})$ one has $f := \mathcal{H}x \in D(J)$ and $x_n := \mathcal{H}_n^{-1} f \in D(J \mathcal{H}_n) = D(A_n)$ and $A_n x_n = Jf = Ax \longrightarrow Ax$ as well as $x_n = \mathcal{H}_n^{-1} f \longrightarrow \mathcal{H}^{-1} f = x$, where for the last convergence we used dominated convergence along with Lemma~\ref{lm:approx-en-density}~(i) and~(ii).)
Combining now~\eqref{eq:stab-thm, step 5, 1} and~\eqref{eq:stab-thm, step 5, 2}, we see by the theorem of Trotter and Kato (Theorem~III.4.8 of~\cite{EnNa}) that $\e^{A_n t} \longrightarrow \e^{A t}$ in the strong operator topology of $X$ as $n \to \infty$ for every $t \in \R^+_0$. So, by~\eqref{eq:stab-thm, step 5, 1},  
\begin{align}
\norm{\e^{A t}}_{X, X} \le M \, \e^{\omega_0 t}
\end{align} 
for every $t \in \R^+_0$, which in view of $\omega_0 < 0$ proves the asserted exponential stability. 
\end{proof}

%STABILISIERUNGSSATZ

With the above theorem at hand, we can now easily prove the following stabilization result. %It has a clear control-theoretic interpretation, see below
See the remarks after the corollary %below 
for a control-theoretic %application-oriented
interpretation of this result and its assumptions.

\begin{cor} \label{cor:stabiliz-result}
Suppose $\mathcal{A}: D(\mathcal{A}) \subset X \to X$ is a first-order port-Hamiltonian operator with energy density $\mathcal{H} \in BV([a,b],\K^{m\times m})$, where $X := L^2([a,b],\K^m)$ is endowed with the $\mathcal{H}$-energy norm $\norm{\cdot}_X$ and where
\begin{align} \label{eq:D(A)-with-zero-bdry-cond}
D(\mathcal{A}) = \big\{ x \in X: \mathcal{H}x \in W^{1,2}((a,b),\K^m) \text{ and } W_{B,1}(\mathcal{H}x)|_{\partial} = 0 \big\}
\end{align}
for some matrix $W_{B;1} \in \K^{(m-k) \times 2m}$ and some $k \in \{1, \dots, m\}$. 
Suppose further that $\mathcal{B}, \mathcal{C}: D(\mathcal{A}) \subset X \to \K^k$ are linear boundary operators given by 
\begin{align} \label{eq:bdry-op-def}
\mathcal{B}x := W_{B,2}(\mathcal{H}x)|_{\partial} 
\qquad \text{and} \qquad
\mathcal{C}x := W_{C}(\mathcal{H}x)|_{\partial} 
\end{align}
with matrices $W_{B,2}, W_C \in \K^{k \times 2m}$ such that the following two conditions are satisfied:
\begin{itemize}
\item[(i)] $\Re\scprd{x,\mathcal{A}x}_X \le (\mathcal{B}x)^* \mathcal{C}x$ for all $x \in D(\mathcal{A})$
\item[(ii)] there exists a constant $\lambda \in (0,\infty)$ such that for $c = a$ or $c = b$ one has
\begin{align*} %\label{eq:state-dominated-by-input+output}
|\mathcal{B}x|^2 + |\mathcal{C}x|^2 \ge \lambda |(\mathcal{H}x)(c)|^2
\qquad (x \in D(\mathcal{A})).
\end{align*}
\end{itemize}
Then for every $\mu \in (0,\infty)$ the operator $A := \mathcal{A}|_{D(A)}$ with domain $D(A) := \{ x \in D(\mathcal{A}): \mathcal{B}x = -\mu \, \mathcal{C}x \}$
%\begin{align*}
%D(A) := \big\{ x \in D(\mathcal{A}): \mathcal{B}x = -\mu \, \mathcal{C}x \big\}
%\end{align*}
generates an exponentially stable contraction semigroup on $X$. 
\end{cor}

\begin{proof}
Choose and fix $\mu \in (0,\infty)$ and define the matrix
\begin{align*}
W := 
\begin{pmatrix}
W_{B,1} \\ W_{B,2}+\mu W_C
\end{pmatrix}
\in 
\K^{m\times 2m}.
\end{align*}
Then the domain of $A$ is of the form~\eqref{eq:stab-thm, domain-ass}
%\begin{align*}
%D(A) = \big\{ x \in X: \mathcal{H}x \in W^{1,2}((a,b),\K^m) \text{ and } W(\mathcal{H}x)|_{\partial} = 0 \big\}
%\end{align*}
and there exists a constant $\kappa \in (0,\infty)$ such that for $c=a$ or $c=b$ the estimate~\eqref{eq:stab-thm, dissip-ass} holds true. Indeed, setting
\begin{align*}
\kappa := \lambda \min\Big\{ \frac{1}{2\mu}, \frac{\mu}{2} \Big\},
\end{align*}
we conclude from our assumptions~(i) and~(ii) that for every $x \in D(A)$
\begin{align*}
\Re \scprd{x,Ax}_X \le \frac{1}{2} (\mathcal{B}x)^* \mathcal{C}x +  \frac{1}{2} (\mathcal{B}x)^* \mathcal{C}x
= -\frac{1}{2\mu} |\mathcal{B}x|^2 - \frac{\mu}{2} |\mathcal{C}x|^2
\le -\kappa |(\mathcal{H}x)(c)|^2. 
\end{align*}
So, the assertion of the corollary follows from the previous theorem.
\end{proof}

%Was bedeutet die Aussage des obigen Stabilisierungskor in kontrolltheoret Sprache?
In control-theoretic terms, the above corollary says that the linear input-output system
\begin{gather}
\dot{x} = \mathcal{A}x \label{eq:evol-eq}\\
u(t) = \mathcal{B}x(t) \qquad \text{and} \qquad y(t) = \mathcal{C}x(t) \label{eq:input/output-cond}
\end{gather}
with %state $x(t)$ at time $t$ and 
control input $u$ and observation output $y$ is exponentially stabilized by means of the negative output-feedback law
\begin{align} \label{eq:output-feedback}
u(t) = -\mu y(t)
\end{align}  
with an arbitrary amplification factor $\mu > 0$. 
%
%Was bedeuten die Voraussetzungen des obigen Stabilisierungskor in kontrolltheoret Sprache? Was implizieren sie (nach geringfuegiger Verschaerfung)?
Condition~(i) of the above corollary means that the input-output system~\eqref{eq:evol-eq}, \eqref{eq:input/output-cond} is impedance-passive in the sense of~\cite{TuWe14}, \cite{Au}. 
Condition~(ii), in turn, means %roughly speaking
that the control input and observation output dominate the value of the state at one of the boundary points ($a$ or $b$). 
Also, if one slightly sharpens the assumptions of the above corollary -- namely by additionally requiring that $\mathcal{H} \in AC([a,b],\K^{m\times m})$ and that $\Re\scprd{x,\mathcal{A}x}_X = (\mathcal{B}x)^* \mathcal{C}x$ for all $x \in D(\mathcal{A})$ (impedance-energy-preservation) -- then the system~\eqref{eq:evol-eq}, \eqref{eq:input/output-cond} is classically approximately observable in infinite time in the sense of~\cite{ScZw18} (Condition~4.9). In fact, this can be proven in exactly the same way as Lemma~4.16 of~\cite{ScZw18}.

\section{Some applications}

In this section, we apply our stabilization result to a vibrating string and a Timoshenko beam.

\begin{ex}
Consider a vibrating string~\cite{Vi}, \cite{JaZw}, \cite{Au}, that is, the transverse displacement $w(t,\zeta)$ of the string at %the horizontal 
position $\zeta \in [a,b]$ evolves according to the partial differential equation
\begin{align} \label{eq:string pde}
\rho(\zeta) \partial_t^2 w(t,\zeta) =  \partial_{\zeta} \big( T(\zeta) \partial_{\zeta}w(t,\zeta) \big)
 \qquad (t \in [0,\infty), \zeta \in [a,b])
\end{align}
(vibrating string equation) and the energy $E_w(t)$ of the string at time $t$ is given by 
\begin{align*} %\label{eq:string energy}
E_w(t) = \frac{1}{2} \int_a^b \rho(\zeta) \big( \partial_t w(t,\zeta) \big)^2 + T(\zeta) \big( \partial_{\zeta} w(t,\zeta) \big)^2 \d \zeta.
\end{align*}
In these equations, $\rho$, $T$ are the mass density and the Young modulus of elasticity of the string and they are assumed to belong to $BV([a,b],\R)$ and to be bounded below and above by positive finite constants. 
%to satisfy
%\begin{align*}
%\ul{m} \le \rho(\zeta), T(\zeta) \le \ol{m} \qquad (\zeta \in [a,b])
%\end{align*}
%for some constants $\ul{m}, \ol{m} \in (0,\infty)$.
Also, assume that the string is clamped at its left end, that is, 
\begin{align} \label{eq:string bdry cond}
\partial_t w(t,a) = 0 \qquad (t \in [0,\infty))
\end{align}  
and that the control input $u(t)$ and observation output $y(t)$ are given respectively by the force and by the velocity at the right end of the string, that is,
\begin{align} \label{eq:string input/output}
u(t) = T(b) \partial_{\zeta} w(t,b)
\qquad \text{and} \qquad
y(t) = \partial_t w(t,b)
\end{align}
for all $t \in [0,\infty)$. With the choices
\begin{align*}
x(t)(\zeta) 
:=
\begin{pmatrix}
\rho(\zeta) \partial_t w(t,\zeta) \\ \partial_{\zeta} w(t,\zeta)
\end{pmatrix},
\qquad %\text{and} \qquad
\mathcal{H}(\zeta) 
:=
\begin{pmatrix}
1/\rho(\zeta) & 0 \\ 0 & T(\zeta)
\end{pmatrix},
\qquad
P_1 := \begin{pmatrix} 0 & 1 \\ 1 & 0 \end{pmatrix}
\end{align*}
and $P_0 := 0 \in \R^{2\times 2}$, the pde~\eqref{eq:string pde} with the boundary condition~\eqref{eq:string bdry cond} takes the form~\eqref{eq:evol-eq}  of a first-order port-Hamiltonian system with~\eqref{eq:D(A)-with-zero-bdry-cond} and with $W_{B,1} \in \R^{1\times 4}$ and, moreover, the in- and output conditions~\eqref{eq:string input/output} take the desired form~\eqref{eq:input/output-cond} with~\eqref{eq:bdry-op-def} and with matrices $W_{B,2}, W_C \in \R^{1\times 4}$. 
It is straightforward to verify that $\mathcal{H}$ is an energy density with $\mathcal{H} \in BV([a,b],\R^{2\times 2})$ %belonging to $BV([a,b],\R^{2\times 2})$ 
and that condition~(i) (even impedance-energy-preservation) and condition~(ii) of Corollary~\ref{cor:stabiliz-result} are satisfied. %$\mathcal{H} \in BV([a,b],\R^{2\times 2})$ and  that this system also satisfies condition~(i) (even impedance-energy-preservation) and condition~(ii) of Corollary~\ref{cor:stabiliz-result}.
%this system is impedance-energy-preserving and that~\eqref{eq:state-dominated-by-input+output} holds true. %with $\eta := b$
%
So, by that corollary, the input-output system~\eqref{eq:string pde}, \eqref{eq:string bdry cond}, \eqref{eq:string input/output} is exponentially stabilized by means of the negative output-feedback law~\eqref{eq:output-feedback} with arbitrary $\mu >0$. 
In the special case of constant Young modulus $T \equiv 1$ and amplification factor $\mu = 1$, the present example reduces to a stabilty result from~\cite{CoZu95} (Theorem~10.1).~$\blacktriangleleft$
\end{ex}

%In the special case of constant Young modulus $T \equiv 1$ and amplification factor $\mu = 1$, the above example reduces to a stabilty result from~\cite{CoZu95} (Theorem~10.1). 

\begin{ex}
Consider a beam modelled according to Timoshenko~\cite{Vi}, \cite{JaZw}, \cite{Au}, that is, the transverse displacement $w(t,\zeta)$ and the rotation angle $\phi(t,\zeta)$ of the beam at %the horizontal 
position $\zeta \in [a,b]$ evolve according to the partial differential equations
\begin{gather} 
\rho(\zeta) \partial_t^2 w(t,\zeta) = \partial_{\zeta} \Big( K(\zeta) \big( \partial_{\zeta}w(t,\zeta) - \phi(t,\zeta) \big) \Big) 
\label{eq:Timoshenko pde, 1}\\
I_r(\zeta) \partial_t^2 \phi(t,\zeta) = \partial_{\zeta} \big( E I(\zeta) \partial_{\zeta} \phi(t,\zeta) \big) + K(\zeta) \big( \partial_{\zeta}w(t,\zeta) - \phi(t,\zeta) \big)
\label{eq:Timoshenko pde, 2}
% \qquad (t \in [0,\infty), \zeta \in (a,b))
\end{gather}
for $t \in [0,\infty), \zeta \in [a,b]$ (Timoshenko beam equations) and the energy $E_{w,\phi}(t)$ of the beam at time $t$ is given by 
\begin{align*} %\label{eq:Timoshenko energy}
E_{w,\phi}(t) &= \frac{1}{2} \int_a^b \rho(\zeta) \big( \partial_t w(t,\zeta) \big)^2 + K(\zeta)  \big( \partial_{\zeta}w(t,\zeta) - \phi(t,\zeta) \big)^2 \\
&\qquad \qquad + I_r(\zeta) \big( \partial_t \phi(t,\zeta) \big)^2 + E I(\zeta) \big( \partial_{\zeta} \phi(t,\zeta)\big)^2  \d \zeta.
\end{align*}
In these equations, $\rho$, $E$, $I$, $I_r$, $K$ are respectively the mass density, the Young modulus, the moment of inertia, the rotatory moment of inertia, and the shear modulus of the beam and they are assumed to belong to $BV([a,b],\R)$ and to be bounded below and above by positive finite constants. Also, assume that the beam is clamped at its left end, that is, 
\begin{align} \label{eq:Timoshenko bdry cond}
\partial_t w(t,a) = 0 \qquad \text{and} \qquad \partial_t \phi(t,a) = 0
\qquad (t \in [0,\infty))
\end{align}  
(velocity and angular velocity at the left endpoint $a$ are zero), and that the control input $u(t)$ is given by the force and the torsional moment at  the right end of the beam and the observation output $y(t)$ is given by the velocity and angular velocity at the right end of the beam, that is,
\begin{align} \label{eq:Timoshenko input/output}
u(t) = \begin{pmatrix} K(b) \big( \partial_{\zeta} w(t,b) - \phi(t,b) \big) \\ E I(b) \partial_{\zeta} \phi(t,b) \end{pmatrix},
\qquad
y(t) = \begin{pmatrix} \partial_t w(t,b) \\ \partial_t \phi(t,b) \end{pmatrix}
\end{align}
for all $t \in [0,\infty)$. With the choices
\begin{align*}
x(t)(\zeta) 
:=
\begin{pmatrix}
\partial_{\zeta} w(t,\zeta) - \phi(t,\zeta) \\
\rho(\zeta) \partial_t w(t,\zeta) \\
\partial_{\zeta} \phi(t,\zeta) \\
I_r(\zeta) \partial_t \phi(t,\zeta)
\end{pmatrix},
\qquad %\text{and} \qquad
\mathcal{H}(\zeta) 
:=
\begin{pmatrix}
K(\zeta) & 0 & 0 & 0 \\
0 & 1/\rho(\zeta) & 0 & 0 \\ 
0 & 0 & EI(\zeta) & 0 \\
0 & 0 & 0 & 1/I_r(\zeta)
\end{pmatrix},
%\qquad
%P_1 := \begin{pmatrix} 0 & 1 \\ 1 & 0 \end{pmatrix}
\end{align*}
and an appropriate choice of $P_1, P_0 \in \R^{4\times 4}$, the pde~\eqref{eq:Timoshenko pde, 1}, \eqref{eq:Timoshenko pde, 2} with the boundary conditions~\eqref{eq:Timoshenko bdry cond} take the form~\eqref{eq:evol-eq} of a first-order port-Hamiltonian system with~\eqref{eq:D(A)-with-zero-bdry-cond} and with $W_{B,1} \in \R^{2\times 8}$ and, moreover, the in- and output conditions~\eqref{eq:Timoshenko input/output} take the desired form~\eqref{eq:input/output-cond} with~\eqref{eq:bdry-op-def} and with matrices $W_{B,2}, W_C \in \R^{2\times 8}$. 
It is straightforward to verify that $\mathcal{H}$ is an energy density with $\mathcal{H} \in BV([a,b],\R^{4\times 4})$ %belonging to $BV([a,b],\R^{4\times 4})$ 
and that condition~(i) (even impedance-energy-preservation) and condition~(ii) of Corollary~\ref{cor:stabiliz-result} are satisfied.
%this system is impedance-energy-preserving and that~\eqref{eq:state-dominated-by-input+output} holds true. %with $\eta := b$
%
So, by that corollary, the input-output system~\eqref{eq:Timoshenko pde, 1}, \eqref{eq:Timoshenko pde, 2}, \eqref{eq:Timoshenko bdry cond}, \eqref{eq:Timoshenko input/output} is exponentially stabilized by means of the negative output-feedback law~\eqref{eq:output-feedback}. $\blacktriangleleft$
\end{ex}

\section*{Acknowledgements}

I would like to thank the German Research Foundation (DFG) for  financial support through the grant ``Input-to-state stability and stabilization of distributed-parameter systems'' (DA 767/7-1).

\begin{small}

\end{small}


\begin{thebibliography}{}

\bibitem{AdFo} R. A. Adams, J. J. F. Fournier: Sobolev spaces. 2nd edition. Elsevier, 2003

\bibitem{AmEs} H. Amann, J. Escher: Analysis I, II, III. Birkh\"auser (2005, 2008, 2009)

\bibitem{Au} B. Augner: Stabilisation of infinite-dimensional port-Hamiltonian systems via dissipative boundary feedback. PhD thesis. Available at http://elpub.bib.uni-wuppertal.de/edocs/dokumente/fbc/mathematik/diss2016/augner/dc1613.pdf

\bibitem{CoZu95} S. Cox, E. Zuazua: \emph{The rate at which energy decays in a string damped at one end.} Indiana Univ. Math. J. \textbf{44} (1995), 545-573

\bibitem{EnNa} K.-J. Engel, R. Nagel: One-parameter semigroups for linear evolution equations. Springer (2000)

\bibitem{Folland:real} G. B. Folland: Real analysis. 2nd edition. Wiley (1999)

\bibitem{HillePhillips} E. Hille, R. S. Phillips: Functional analysis and semi-groups. American Mathematical Society Colloquium Publications (1957)

\bibitem{JaZw} B. Jacob, H. Zwart: Linear port-Hamiltonian systems on infinite-dimensional spaces. Birkh\"auser (2012)

\bibitem{JaMoZw15} B. Jacob, K. Morris, H. Zwart: \emph{$C_0$-semigroups for hyperbolic partial differential equations on a one-dimensional spatial domain.} J. Evol. Equ. \textbf{15} (2015), 493-502

\bibitem{Rudin:real-complex} W. Rudin: Real and complex analysis. 3rd edition. McGraw-Hill (1987)

\bibitem{ScZw18}  J. Schmid, H. Zwart: \emph{Stabilization of port-Hamiltonian systems by nonlinear boundary control in the presence of disturbances.} arXiv:1804.10598 (2018)

\bibitem{Si20a} W. Sierpi\'{n}ski: \emph{Sur un probl\`{e}me concernant les ensembles m\'{e}surables superficiellement.} Fund. Math.~\textbf{1} (1920), 112-115 

\bibitem{Si20b} W. Sierpi\'{n}ski: \emph{Sur les rapports entre l'existence des intégrales $\int_0^1f(x,y)dx$, $\int_0^1f(x,y)dy$ et $\int_0^1dx\int_0^1f(x,y)dy$.} Fund. Math.~\textbf{1} (1920), 142-147

\bibitem{TuWe14} M. Tucsnak, G. Weiss: \emph{Well-posed systems -- the LTI case and beyond.} Automatica \textbf{50} (2014), 1757-1779 

\bibitem{Vi} J. Villegas: A port-Hamiltonian approach to distributed-parameter systems. Ph.D. thesis, Universiteit Twente (2007)

\bibitem{ViZwLeMa09} J. Villegas, H. Zwart, Y. Le Gorrec, B. Maschke: \emph{Exponential stability of a class of boundary control systems.} IEEE Trans. Autom. Contr. \textbf{54} (2009), 142-147 

\end{thebibliography}
\end{document}